\documentclass[reqno]{amsart}
\usepackage{fullpage}

\usepackage{cite}

\usepackage{amsmath, amssymb, amsthm, mathrsfs, graphicx, float}
\usepackage[T1]{fontenc}
\usepackage[initials,msc-links]{amsrefs}
\usepackage[colorinlistoftodos]{todonotes}
\usepackage{tikzsymbols}
\usepackage{mathtools}
\usepackage{hyperref}
\usepackage{nicefrac}
\usepackage{tikz-cd}
\usepackage{refcount}
\usepackage{comment}

\usepackage{mathtools}

\newtheorem{theorem}{Theorem}[section]
\newtheorem{lemma}[theorem]{Lemma}
\newtheorem{proposition}[theorem]{Proposition}
\newtheorem{conjecture}[theorem]{Conjecture}

\theoremstyle{remark}

\newtheorem{remark}[theorem]{Remark}
\theoremstyle{definition}

\newcommand{\N}{\mathbb{N}}

\newcommand{\F}{\mathbb{F}}
\newcommand{\f}{\mathcal{F}}
\newcommand{\I}{\mathcal{I}}

\renewcommand{\P}{\mathcal{P}}
\newcommand{\e}{\varepsilon}

\newcommand{\ga}{\gamma}
\newcommand{\ze}{\zeta}
\newcommand{\m}{\rule[0.25em]{4pt}{0.5pt}}
\newcommand{\ignore}[1]{}

\DeclarePairedDelimiter\floor{\lfloor}{\rfloor}

\title{On the Erd\H{o}s primitive set conjecture in function fields}
\author{Andr\'{e}s G\'{o}mez-Colunga, Charlotte Kavaler, Nathan McNew, and Mirilla Zhu}

\begin{document}
\maketitle 

\begin{abstract}
Erd\H{o}s proved that $\mathcal{F}(A) := \sum_{a \in A}\frac{1}{a\log a}$ converges for any primitive set of integers $A$ and later conjectured this sum is maximized when $A$ is the set of primes. Banks and Martin further conjectured that  $\mathcal{F}(\mathcal{P}_1) > \ldots > \mathcal{F}(\mathcal{P}_k) > \mathcal{F}(\mathcal{P}_{k+1}) > \ldots$, where $\mathcal{P}_j$ is the set of integers with $j$ prime factors counting multiplicity, though this was recently disproven by Lichtman.  We consider the corresponding problems over the function field $\mathbb{F}_q[x]$, investigating the sum $\mathcal{F}(A) := \sum_{f \in A} \frac{1}{\textnormal{deg} f \cdot q^{\textnormal{deg} f}}$. We establish a uniform bound for $\mathcal{F}(A)$ over all primitive sets of polynomials $A \subset \mathbb{F}_q[x]$ and conjecture that it is maximized by the set of monic irreducible polynomials. We find that the analogue of the Banks-Martin conjecture is false for $q = 2$, $3$, and $4$, but we find computational evidence that it holds for $q > 4$.   
\end{abstract}
\section{Introduction}
A primitive set is one in which no element of the set divides another. In 1935, Erd\H{o}s \cite{erdos35} proved that for any primitive set of positive integers $A\neq\{1\}$,
$$\f(A) \ \coloneqq \ \sum_{a \in A} \frac{1}{a\log a} \ < \ \infty.$$
In 1988, Erd\H{o}s conjectured that the primitive set which maximizes this sum is the set of primes. 
\begin{conjecture}[Erd\H{o}s]
Let $\P$ denote the set of prime numbers. For all primitive sets of positive integers $A \neq \{1\}$, 
\begin{equation}\sum_{a \in A} \frac{1}{a \log a} \ \leq \ \sum_{p \in \P} \frac{1}{p\log p}. \label{eq:erdsum} \end{equation}
\end{conjecture}
\noindent While this conjecture remains open, significant progress has been made. In 1991, Zhang \cite{zhang91} showed that the conjecture holds for all primitive sets containing no element with more than four prime factors counted with multiplicity. Two years later, Erd\H{o}s and Zhang \cite{erdzhang94} showed that $\f(A)<1.84$ for any primitive set; this was improved last year by Lichtman and Pomerance \cite{lichtmanPom19} to $\f(A)  \ \leq \  e^\gamma = 1.781072\ldots$.  For comparison, we know due to Cohen \cite{cohen} that $\f(\P) = 1.636616\ldots$.

In 2013, Banks and Martin \cite{banksMartin13} proposed a related conjecture concerning the Erd\H{o}s sum of primitive sets with a fixed number of prime factors. 
\begin{conjecture}\label{banks-martin}\textnormal{(Banks, Martin)}
Let $\P_k$ be the set of natural numbers with exactly $k$ prime factors counted with multiplicity and let $\f$ be the Erd\H{o}s sum in the integers. Then 
$$\f(\P_1) \ > \  \f(\P_2) \ > \  \ldots \ > \  \f(\P_k) \ > \  \f(\P_{k+1})\ldots.$$
\end{conjecture}
\noindent Taken together with a theorem of Zhang \cite{zhang91}, results of Bayless, Kinlaw, and Klyve \cite{bkk19} show that $\f(\P_1) > \f(\P_2) > \f(\P_3)$. Just this year, however, Lichtman \cite{lichtman} showed that the conjectured inequality fails to hold for all $k$, and that $\f(\P_k)$ in fact attains a global minimum at $k = 6$.

In this paper, we examine analogues of these conjectures for the function field $\F_q[x]$. Here, the natural parallel of the Erd\H{o}s sum \eqref{eq:erdsum} is
\[ \f(A) \ \coloneqq \ \sum_{a \in A} \frac{1}{q^{\deg a} \deg a},\]
which we conjectured in \cite{us19} is maximized by the set $\I_q \subset \F_q[x]$ of monic irreducible polynomials.

In Section 2, we estimate $\f(\I_q)$ and show that it approaches $\frac{\pi^2}{6} = 1.644930\ldots$ as $q \to \infty$. We then establish effective bounds for the function field analogue of Mertens' third theorem in Section 3, which we use to compute an upper bound for $\f(A)$ over all primitive sets $A\subset\F_q[x]$ in Section 4. When $3 \leq q \leq 19$, we obtain a bound of $e^\gamma$ just as in the integer case, and when $q > 19$, we obtain a bound of $e^{\gamma - 1} + \frac{\pi^2 - 3}{6} = 1.800153\ldots$.  In the case where $q = 2$, we show that $\f(A) <  1+ \frac{e^\ga}{2} = 1.890536\ldots$.

In Sections 5 and 6, we consider the function field analogue of the Banks-Martin conjecture. Letting $\mathcal{I}_{k,q}$ be the set of monic polynomials in $\F_q[x]$ with $k$ irreducible factors, we demonstrate that the infinite chain of inequalities
$$\f(\mathcal{I}_{1,q}) \ > \  \f(\mathcal{I}_{2,q}) \ > \  \ldots \ > \  \f(\mathcal{I}_{k,q}) \ > \  \f(\mathcal{I}_{k+1,q})\ldots$$
fails to hold when $q = 2$, $3$, or $4$. However, we show that for each $k$, there exists a $q_k$ such that
\[\f(\mathcal{I}_{1,q}) \ > \  \f(\mathcal{I}_{2,q}) \ > \  \ldots \ > \  \f(\mathcal{I}_{k,q})\]
for all $q \geq q_k$, and furthermore, that $q_k = O(k^24^k)$. We also present an approach to efficiently compute $\f(\mathcal{I}_{k,q})$ with high precision, providing numerical evidence that the Banks-Martin conjecture in $\F_q[x]$ may hold in full generality when $q \geq 5$.

For the remainder of this paper, we denote the degree of a polynomial $f \in \F_q[x]$ by $\deg f$ and write $||f|| = q^{\deg f}$ for the norm of $f$. Following the conventions we established in \cite{us19}, we restrict our attention to primitive subsets of monic polynomials and exclude the set $\{1\}$ from consideration. 

\section{Counting Irreducibles in $\F_q[x]$}
We begin by evaluating the Erd\H{o}s sum over the monic irreducibles $\I_q \subset \F_q[x]$. Letting $\pi'_q(n)$  denote the number of degree $n$ irreducibles in $\F_q[x]$, we rewrite our sum as
\[\f(\I_q) \ = \ \sum_{n = 1}^{\infty} \frac{\pi'_q(n)}{nq^n}. \]
The numerators of this sum can be expressed in terms of the M\"{o}bius function $\mu$ using Gauss' formula 
\[\pi'_q(n) \ = \ \frac{1}{n} \sum_{d|n} q^d \mu\left(\frac{n}{d} \right),\]
which allows us to obtain bounds on $\pi'_q(n)$ and $\f(\I_q)$. 

 \begin{proposition}\label{lower bound of pi'} 
            \[\frac{q^n}{n} - \Big(\frac{q}{q-1}\Big)\frac{q^{n/2}}{n} \ \leq \ \pi'_q(n) \ \leq \ \frac{q^n}{n}.\]
            \end{proposition}
            
            \begin{proof}
            The upper bound is a known result whose proof can be found in \cite{MR2712407}. The lower bound is immediate when $n = 1$, so we will consider the case where $n > 1$. We know from Gauss' formula that
            $$\frac{q^n}{n} - \pi'_q(n) \ = \ \frac{q^n}{n} - \frac{1}{n} \sum_{d|n} q^d \mu \left( \frac{n}{d} \right) \ = \   \frac{q^{n/p'}}{n} - \frac{1}{n} \sum_{\substack{d|n,\\ d < n/p'}} q^d \mu\left(\frac{n}{d}\right),$$
            where $p'$ is the smallest prime factor of $n$. This expression is at most
            \[\frac{q^{n/p'}}{n} + \frac{1}{n} \sum_{d=1}^{n/p' - 1} q^d \ = \  \frac{q^{n/p'}}{n} + \frac{1}{n}\left( \frac{q^{n/p'} - q}{q-1} \right) \ \leq \  \frac{q^{n/2}}{n} + \frac{1}{n}\left( \frac{q^{n/2} - q}{q-1}\right) \ = \  \Big(\frac{q}{q-1}\Big) \Big(\frac{q^{n/2}-1}{n}\Big), \] 
            which gives that
            \[\pi'_q(n) \  \geq \ \frac{q^n}{n} - \Big(\frac{q}{q-1}\Big) \frac{q^{n/2}}{n}. \qedhere\]
        \end{proof}
        
        \begin{proposition}\label{prop.recipsum}
\[\frac{\pi^2}{6} - \frac{q}{q-1} \textnormal{Li}_2 \bigg( \frac{1}{\sqrt{q}} \bigg)
\ \leq \ \f(\mathcal{I}_q) \ 
\leq \ \frac{\pi^2}{6},\]
where $\textnormal{Li}_2(x)$ is the dilogarithm $\textnormal{Li}_2(x)=\sum_{k = 1}^{\infty} \frac{x^k}{k^2}$. In particular, $\f(\I_q) \to \frac{\pi^2}{6}$ as $q \to \infty$. 
\end{proposition}

\begin{proof}
These bounds are a consequence of Proposition \ref{lower bound of pi'}. For the upper bound, we have
\[\sum_{n = 1}^{\infty} \frac{\pi'_q(n)}{nq^n} \ \leq \ \sum_{n = 1}^{\infty} \frac{1}{n^2} \ = \ \frac{\pi^2}{6}, \]
and for the lower bound, we have
\[\sum_{n = 1}^{\infty} \frac{\pi'_q(n)}{nq^n} \ \geq \ \sum_{n = 1}^{\infty} \frac{1}{n^2} -\left(\frac{q}{q-1}\right) \frac{1}{n^2q^{n/2}} \ = \ \frac{\pi^2}{6} - \frac{q}{q-1} \textnormal{Li}_2 \bigg( \frac{1}{\sqrt{q}} \bigg).\]
As $q \to \infty$, $\textnormal{Li}_2\left(\frac{1}{\sqrt{q}}\right) \to 0$, so the lower bound for $\f(\I_q)$ converges to the upper bound. 
\end{proof}

In the following proposition, we show that the value of $\f(\I_q)$ increases monotonically with $q$, which implies that $\f(\I_2)$ is a lower bound on $\f(\I_q)$ for any $q$. The lower bound obtained by computing this sum (see Section 2.1) is strictly better than
the lower bound in Proposition \ref{prop.recipsum} for all $q < 37$.

\begin{proposition}\label{prop.hello} For any prime powers $q _1 < q_2$, $\mathcal{F}(\mathcal{I}_{q_1})<\mathcal{F}(\mathcal{I}_{q_2})$.
\end{proposition}

\begin{proof}
The inequality can be verified computationally for $q_1 = 2$ and $q_2 = 3$. To address the remaining cases, we will show that each term $\frac{\pi_q'(n)}{nq^n}$ of $\f(\I_q)$ is strictly increasing in $q$ when $q \geq 3$. By Gauss' formula, this is equivalent to showing that 
\[\frac{1}{q^n}\sum_{d|n} q^d \mu \left( \frac{n}{d} \right)\] 
is increasing in $q$. The derivative of this expression with respect to $q$ is
\[\frac{q^n\sum_{d|n} dq^{d-1} \mu \big( \frac{n}{d} \big) - n q^{n-1} \sum_{d|n} q^d \mu \big( \frac{n}{d} \big) }{q^{2n}}
\ = \ \frac{\sum_{d|n} (d-n)q^{d} \mu \big( \frac{n}{d} \big)}{q^{n+1}}.\]
To show it is positive, we first note that $\sum_{d|n} (d-n) q^{d}\mu \big( \frac{n}{d} \big)$ is a polynomial in $q$ whose leading nonzero term is $(n - \frac{n}{p'})q^{n/p'}$, where $p'$ is the smallest prime factor of $n$. Since $\mu(m) \leq 1$ for all $m$, this polynomial can be bounded below by
\[\left(n -\frac{n}{p'}\right) q^{n/p'} + \sum_{d = 1}^{n/p' - 1}(d - n)q^{d}.\] 
This expression in turn is at least
\begin{align*}\left(n -\frac{n}{p'}\right) q^{n/p'} -\sum_{d = 1}^{n/p' - 1}nq^{d} 
\ &= \ \left(n -\frac{n}{p'}\right) q^{n/p'} - n\left( \frac{q^{n/p'}-1}{q-1}\right) + n \\
&= \ n q^{n/p'} \Big( \frac{p' - 1}{p'} - \frac{1}{q-1} \Big) + \frac{n}{q-1} + n \ \geq \ \frac{n}{q-1} + n,
\end{align*}
where the last inequality holds because $\frac{p' - 1}{p'} - \frac{1}{q-1}$ is nonnegative for all $p' \geq 2$ and $q \geq 3$. It follows that the derivative is positive, which means that $\f(\I_q)$ is strictly increasing for $q \geq 3$. 
\end{proof}

\subsection{Numerical note}  Even though a closed formula for $\f(\I_q)$ seems elusive, it is surprisingly easy to compute its value to very high precision for any fixed value of $q$.  Suppose we have computed a partial sum of $\f(\I_q)$, 
\[S_{N,q} \ \coloneqq \ \sum_{n=1}^N \frac{\pi'_q(n)}{nq^n}.\] 
We estimate the remainder of this sum as 
\begin{align}
    \f(\I_q) - S_{N,q} \ &= \ \sum_{n=N+1}^\infty \frac{\pi'_q(n)}{nq^n} 
    \ = \ \sum_{n=N+1}^\infty \left( \frac{1}{n^2} - \frac{\frac{q^n}{n} - \pi'_q(n)}{nq^n}\right) \nonumber \\
    &= \ \zeta(2) - \sum_{n=1}^n \frac{1}{n^2} - \sum_{n=N+1}^\infty \left(\frac{\frac{q^n}{n} - \pi'_q(n)}{nq^n}\right) \label{eq:irredsumremainder}.
\end{align}
From Proposition \ref{lower bound of pi'}, we have \[0 \ \leq \ \frac{q^n}{n} - \pi'_q(n) \ \leq \ \left( \frac{q}{q-1}\right)\frac{q^{n/2}}{n},\]
and so 
\begin{align*}
    0 \ \leq \ \sum_{n=N+1}^\infty \left(\frac{\frac{q^n}{n} - \pi'_q(n)}{nq^n}\right) \ &\leq \ \frac{q}{q-1}\sum_{n=N+1}^\infty \frac{1}{n^2q^{n/2}} \ < \ \frac{q}{N^2(q-1)}\left(\frac{q^{-(N+1)/2}}{1-q^{-1/2}}\right)  \ < \ \frac{5 q^{-N/2}}{N^2}
\end{align*}
for $q\geq 2$.  Using this in \eqref{eq:irredsumremainder} gives the bounds 
\[S_{N,q} + \zeta(2) - \sum_{n=1}^N \frac{1}{n^2} - \frac{5 q^{-N/2}}{N^2} \ \leq \ \f(\I_q) \ \leq \ S_{N,q} + \zeta(2) - \sum_{n=1}^N \frac{1}{n^2}.\]

When $q=2$, taking $N=70,000$ is sufficient to compute the value of $\f(\mathcal{I}_2) = 1.4676602238442289268\ldots$ to over 10,000 digits accuracy in a few seconds, and this converges even faster for larger values of $q$.

\section{Bounds for the Mertens Product}
In \cite{us19}, as part of our proof that the Erd\H{o}s sum converges for all primitive sets, we used the Sieve of Erastosthenes to show that the density of multiples of $f$ with no irreducible factors of smaller degree is 
\[\frac{1}{||f||} \ \prod_{\substack{p\in \I_q \\ \deg p \leq D(f)}} \left(1 - \frac{1}{||p||}\right),\] 
where $D(f)$ denotes the largest degree of an irreducible factor of $f$. We were then able to bound this expression using an analogue of Mertens' third theorem in function fields--a special case of Theorem 3 in \cite{MR1700882}.
\begin{theorem}\label{thm.mertens}
\begin{equation}
    \prod_{\substack{p \in \I_q \\ \deg p \leq n}} \left(1 - \frac{1}{||p||} \right) \ \sim \ \frac{1}{e^{\gamma}n}, \label{eq:mertensprod}
\end{equation}
where $\gamma = 0.577215\ldots$ is the Euler-Mascheroni constant.
\end{theorem}

In order to obtain a numerical upper bound for $\f(A)$, we'll need to establish more precise bounds for the Mertens product \eqref{eq:mertensprod}.  If we take the natural logarithm of this product, we obtain
\[ \sum_{i=1}^n \pi'_q(i) \log \Big( 1 - \frac{1}{q^i} \Big) \ = \ - \sum_{i=1}^n \pi'_q(i) \left(\sum_{k=1}^\infty \frac{1}{kq^{ik}} \right).\]
Below we have written out the first six terms of this summation. Notice that the sum of the constant terms from each expression form a partial sum of the harmonic series, and that partial cancellation occurs in the coefficients of other powers of $q$. In particular, the sums for the coefficients of $\frac{1}{q^j}$ are zero for $1\leq j \leq 3$; the terms perfectly cancel out.
\[\begin{array}{ccccccccccccccccc}
\m\pi'_q(1)\log\Big(1-\frac{1}{q^1}\Big) & = & q \sum\limits_{k=1}^\infty \frac{1}{kq^{k}} & = &
1 & \!+\! & \frac{1}{2q} & \!+\! & \frac{1}{3q^2} & \!+\! & \frac{1}{4q^3} & \!+\! & \frac{1}{5q^4} & \!+\! & \frac{1}{6q^5} & \!+\! & \cdots\\
\m\pi'_q(2)\log\Big(1-\frac{1}{q^2}\Big) & = & \frac{q^2-q}{2} \sum\limits_{k=1}^\infty \frac{1}{kq^{2k}} & = &
\frac{1}{2} & \!+\! & \frac{\m1}{2q} & \!+\! & \frac{1}{4q^2} & \!+\! & \frac{\m1}{4q^3} & \!+\! & \frac{1}{6q^4} & \!+\! & \frac{\m1}{6q^5} & \!+\! & \cdots\\
\m\pi'_q(3)\log\Big(1-\frac{1}{q^3}\Big) & = & \frac{q^3-q}{3} \sum\limits_{k=1}^\infty \frac{1}{kq^{3k}} & = &
\frac{1}{3} & \!+\! & 0 & \!+\! & \frac{\m1}{3q^2} & \!+\! & \frac{1}{6q^3} & \!+\! & 0 & \!+\! & \frac{\m1}{6q^5} & \!+\! & \cdots\\
\m\pi'_q(4)\log\Big(1-\frac{1}{q^4}\Big) & = & \frac{q^4-q^2}{4} \sum\limits_{k=1}^\infty \frac{1}{kq^{4k}} & = &
\frac{1}{4} & \!+\! & 0 & \!+\! & \frac{\m1}{4q^2} & \!+\! & 0 & \!+\! & \frac{1}{8q^4} & \!+\! & 0 & \!+\! & \cdots\\
\m\pi'_q(5)\log\Big(1-\frac{1}{q^5}\Big) & = & \frac{q^5-q}{5} \sum\limits_{k=1}^\infty \frac{1}{kq^{5k}} & \!=\! &
\frac{1}{5} & \!+\! & 0 & \!+\! & 0 & \!+\! & 0 & \!+\! & \frac{\m1}{5q^4} & \!+\! & \frac{1}{10q^5} & \!+\! & \cdots\\
\m\pi'_q(6)\log\Big(1-\frac{1}{q^6}\Big) & = & \frac{q^6-q^3-q^2+q}{6} \sum\limits_{k=1}^\infty \frac{1}{kq^{6k}}\! & \!=\! &
\frac{1}{6} & \!+\! & 0 & \!+\! & 0 & \!+\! & \frac{\m1}{6q^3} & \!+\! & \frac{\m1}{6q^4} & \!+\! & \frac{1}{6q^5} & \!+\! & \cdots
\end{array}\]

In the following lemma, we show that for all $n$, this same cancellation occurs for each $j \in \left[1, \left\lfloor\frac{n}{2}\right\rfloor \right]$. By bounding the contribution from terms $\frac{1}{q^j}$ with $j > \frac{n}{2} $, we obtain bounds for $\sum \pi'_q(i) \log \Big( 1 - \frac{1}{q^i} \Big)$ in terms of partial sums of the harmonic series. 

\begin{lemma} \label{2.0.5 Lemma}
\[\left( 1 - \frac{1}{2(q-1)q^{\floor{n/2}}} \right) \sum_{i=1}^n \frac{1}{i} \ \leq \ \left| \sum_{i=1}^n \pi'_q(i) \log \left( 1 - \frac{1}{q^i} \right) \right| \ \leq \ \left( 1 + \frac{1}{2(q-1)q^{\floor{n/2}}} \right) \sum_{i=1}^n \frac{1}{i}.\]
\end{lemma}

\begin{proof}
To simplify our calculations, we define 
\[\nu_i(d) = \begin{cases}\mu\left(\frac{i}{d} \right) & d|i \\
0 & \textnormal{otherwise}
\end{cases}\]
so that our formula for $\pi'_q(i)$ can be written as 
\[\frac{1}{i} \sum_{d|i} q^d \mu\left( \frac{i}{d}\right) \ = \ \frac{1}{i} \sum_{d = 1}^i q^d \nu_i(d).\]
Substituting this expression for $\pi’(i)$ and expanding each logarithm as a Taylor series gives
\begin{align}
\left| \sum_{i=1}^n \pi’(i) \log \Big( 1 - \frac{1}{q^i} \Big) \right|
& \ = \ \sum_{i=1}^n \left( \frac{1}{i} \sum_{d=1}^i q^d \nu_i(d) \right) \left( \sum_{k=1}^{\infty} \frac{1}{kq^{ik}} \right) \nonumber \\
&= \ \sum_{i=1}^n \sum_{d=1}^i \sum_{k=1}^{\infty} \left(\frac{\nu_i(d)}{ik} \cdot \frac{1}{q^{ik-d}}\right). \label{eq:triple sum}
\end{align}
Since $d \leq ik$ for all terms in this triple sum, it can be written as a power series of the form $\sum_{j=0}^\infty c_j \cdot \frac{1}{q^j}$
for some coefficients $c_j$. In particular, it will be the case that
\[c_j \ = \ \sum_{d=1}^n \frac{1}{j+d} \sum_{\substack{r | \frac{j+d}{d} \\ r \leq n/d}} \mu(r).\]

To see why this is true, note that the terms in \eqref{eq:triple sum} which contribute to $c_j$ are exactly those for which $ik = j+d$. Since $\nu_i(d) = 0$ for all $d > i$, we can extend the sum over $d$ to include values up to $d=n$. Furthermore, because the sum has finitely many nonzero terms, we can interchange the order of summation so that
\[\sum_{i=1}^n \sum_{d=1}^i \sum_{k=1}^\infty \left( \frac{\nu_i(d)}{ik} \cdot \frac{1}{q^{ik-d}} \right)
\ = \  \sum_{d=1}^n \sum_{k=1}^\infty \sum_{i=1}^n \left( \frac{\nu_i(d)}{ik} \cdot \frac{1}{q^{ik-d}}\right).\]
For any fixed $d$, its contribution to $c_j$ is
\[\sum_{\substack{i,k \\ ik = j+d \\ i \leq n}} \frac{\nu_i(d)}{ik}
\ = \ \frac{1}{j+d} \sum_{\substack{i,k \\ ik = j+d \\ i \leq n}} \nu_i(d)
\ = \  \frac{1}{j+d}\sum_{\substack{ i | (j+d) \\ i \leq n}} \nu_i(d) \ = \  \frac{1}{j+d} \sum_{\substack{r | \frac{j+d}{d} \\ r \leq n/d}} \mu(r),\]
where we have made the substitution $r = \frac{i}{d}$ and used the definition of $\nu$ in the last equality. Summing over all $d$ gives us the desired expression for $c_j$. 

For the specific case of $j = 0$,
\[c_0 \ = \ \sum_{d=1}^n \frac{1}{d} \sum_{r | 1} \mu(r) \ = \ \sum_{d=1}^n \frac{1}{d}.\]
Now consider the case in which $j \in [1, \frac{n}{2}]$. If $r | \frac{j+d}{d}$ then $\frac{j+d}{d}$ is an integer, which implies that $d \leq j$. It follows that $j+d \leq 2j \leq n$, so $r \leq \frac{n}{d}$ whenever $r | \frac{j+d}{d}$. As a result,
\[c_j \ = \ \sum_{d=1}^n \frac{1}{j+d} \sum_{r | \frac{j+d}{d}} \mu(r)
\ = \ \sum_{d=1}^n \frac{1}{j+d}\cdot 0 \ = \ 0,\]
where we have used the fact that the sum of $\mu(r)$ over all divisors of $\frac{j+d}{d}$ equals zero. Hence we can rewrite our expression as follows: \begin{align}
\left| \sum_{i=1}^n \pi'_q(i) \log \Big( 1 - \frac{1}{q^i} \Big) \right| \ = \ \sum_{j=0}^\infty c_j \cdot\frac{1}{q^j} \ = \ \sum_{d=1}^n \frac{1}{d} \  +
\sum_{j=\floor{n/2}+1}^\infty c_j \cdot\frac{1}{q^j}. \label{eq:logsumterms}
\end{align}
Our final task will be to bound the last summation. To do so, observe that 
\[\bigg| \sum_{r \in S} \mu(r) \bigg| \ \leq \ \frac{j+d}{2d}\]
whenever $S$ is a subset of divisors of $\frac{j+d}{d}$. This follows from the fact that $\mu(r)$ can only take on values of $1$, $-1$, or $0$ for at most $\frac{j+d}{d}$ different values of $r$, and that the sum of $\mu(r)$ over all divisors of $\frac{j+d}{d}$ equals zero. In particular, this holds when $S$ is the subset of divisors that are at most $\frac{n}{d}$, so
\[ |c_j| \ = \ \sum_{d=1}^n  \frac{1}{j+d} \Bigg| \sum_{\substack{r | \frac{j+d}{d} \\ r \leq  n/d }} \mu(r) \Bigg| \leq
\  \sum_{d=1}^n \frac{1}{j+d}\cdot\frac{j+d}{2d} \ = \ \sum_{d=1}^n \frac{1}{2d}.\]
 It follows that
\[ \left| \sum_{j = \floor{n/2}+1}^\infty c_j \cdot \frac{1}{q^j} \right| \ \leq  \ \sum_{j = \floor{n/2}+1}^\infty \bigg| c_j \cdot \frac{1}{q^j} \bigg| \ \leq \ \sum_{j = \floor{n/2}+1}^\infty \frac{1}{q^j} \sum_{d=1}^n \frac{1}{2d} \ = \ \frac{1}{2(q-1)q^{\floor{n/2}}} \sum_{d=1}^n \frac{1}{d}\]
which, along with equation \eqref{eq:logsumterms}, implies that
\[\Big( 1 - \frac{1}{2(q-1)q^{\floor{n/2}}} \Big) \sum_{i=1}^n \frac{1}{i} \ \leq \ \bigg| \sum_{i=1}^n \pi'_q(i) \log \Big( 1 - \frac{1}{q^i} \Big) \bigg| \ \leq \ \Big( 1 + \frac{1}{2(q-1)q^{\floor{n/2}}} \Big) \sum_{i=1}^n \frac{1}{i}.\qedhere\]
\end{proof}

\begin{proposition} \label{prop.mertensub}
 \[\prod_{\substack{p \in \mathcal{I}_q \\ \deg p \leq n}} \left(1 - \frac{1}{||p||} \right) \ \leq \ \frac{1}{e^\gamma n}.\]
\end{proposition}
\begin{proof}
From Lemma $\ref{2.0.5 Lemma}$, we have
\[- \sum_{i=1}^n \pi'_q(i) \log \left( 1 - \frac{1}{q^i} \right) \ \geq \ \sum_{i=1}^n \frac{1}{i} \left(1 - \frac{1}{2(q-1)q^{n/2}} \right).\]
For the harmonic number $\sum_{i=1}^{n} \frac{1}{i}$, P\'{o}lya and Szerg\H{o} \cite{harmonic} give a lower bound of 
\[\log n + \gamma + \frac{1}{2n} - \frac{1}{8n^2},\]
which we can substitute into our inequality to obtain 
\[\sum_{i=1}^n \pi'_q(i) \log \left( 1 - \frac{1}{q^i} \right) \ \leq \ -\log n - \gamma - \frac{1}{2n} + \frac{1}{8n^2} + \frac{1}{2(q-1)q^{n/2}}\left(\log n + \gamma + \frac{1}{2n} - \frac{1}{8n^2} \right). \]
We claim that in all but finitely many cases,
\[- \frac{1}{2n} + \frac{1}{8n^2} + \frac{1}{2(q-1)q^{n/2}}\left(\log n + \gamma + \frac{1}{2n} - \frac{1}{8n^2} \right) \ \leq \ 0.\]
First consider the case of $q=3$. It can be calculated that the derivative
of the expression with respect to $n$ is negative for all $n \geq 3$, so the expression is decreasing in $n$. The inequality can then be verified computationally for $n \leq 3$, so the inequality holds for all $n$ when $q = 3$. Furthermore, the expression is decreasing in $q$, so the fact that the inequality holds for $q = 3$ implies that it holds for all $q \geq 3$. Setting $q=2$ and taking the derivative with respect to $n$, it can be shown that the inequality also holds for $q = 2$ when $n \geq 10$. Hence in all of these cases
\[\sum_{i=1}^n \pi'_q(i) \log \left( 1 - \frac{1}{q^i} \right) \ \leq \ - \log n - \gamma, \]
and so 
\[\prod_{\substack{p \in \mathcal{I}_q \\ \deg p \leq n}} \left(1 - \frac{1}{||p||} \right) \  \leq \ \frac{1}{e^\gamma n}.\]
The remaining cases in which $q=2$ and $n \leq 9$ can be verified computationally to complete the proof.
\end{proof}

\begin{proposition} \label{prop.mertenslb}
Except in the case $q = 2$ and $n = 1$,

\[\prod_{\substack{p \in \mathcal{I}_q \\ \deg p \leq n}} \Big( 1 - \frac{1}{||p||} \Big) \ > \ \frac{1}{e^\ga(n+1)}.\]
\end{proposition}

\begin{proof}
Once again from Lemma \ref{2.0.5 Lemma}, we have
\[-\sum_{i=1}^n \pi'_q(i) \log \Big( 1 - \frac{1}{q^i} \Big) \ \leq \ -\Big( 1 + \frac{1}{2(q-1)q^{\floor{n/2}}} \Big) \sum_{i=1}^n \frac{1}{i}.\]
From Young \cite{harmonic2}, we know that $\sum_{i=1}^n \frac{1}{i}$ is bounded above by 
\[\log(n+1)  + \ga - \frac{1}{2n+2},\]
so our inequality becomes
\begin{align*} \sum_{i=1}^n \pi'_q(i) \log \Big( 1 - \frac{1}{q^i} \Big)  \ &\geq \ -\log(n+1)  - \ga + \frac{1}{2n+2} - \frac{1}{2(q-1)q^{\floor{n/2}}} \left( \log(n+1) + \ga - \frac{1}{2n+2}\right).
\end{align*}
If we can show that the sum of the last two terms is nonnegative, then we will have 
\[ \sum_{i=1}^n \pi'_q(i) \log \left( 1 - \frac{1}{q^i} \right) \ \geq \ -\log (n+1) - \gamma,\]
upon which exponentiating both sides gives the desired inequality. 

We first prove this is true for $q \geq 4$. Our expression
\[\frac{1}{2n+2} - \frac{1}{2(q-1)q^{\floor{n/2}}} \left( \log(n+1) + \ga - \frac{1}{2n+2} \right)\]
is increasing with respect to $q$, so it suffices to consider the case when $q = 4$. Because $4^{\floor{n/2}} \geq 2^{n-1}$, we only need to demonstrate that
\[\frac{1}{3 \cdot 2^n}\left(\log(n+1) + \ga - \frac{1}{2n+2} \right) \ \leq \ \frac{1}{2n+2}.\]
This inequality can be computationally verified for $n = 1$. For $n \geq 2$, we have $\log(n+1) + \ga - \frac{1}{2n+2} < n$, so it suffices to show
\[\frac{n}{3 \cdot 2^n} \ \leq \ \frac{1}{2n+2},\]
or equivalently, $0 \leq 3 \cdot 2^n - 2n^2 - 2n$. The right hand side equals zero when $n = 2$ or $n = 3$, and its derivative $3\cdot 2^n \log 2 - 4n + 2$ is positive for $n \geq 3$, so the inequality is true for all $n$ when $q \geq 4$. 

Similar analytic arguments can be used to show that the inequality is true for $q=3$ when $n \geq 8$ and for $q=2$ when $n \geq 18$, and the remaining cases can be checked through direct computation.
\end{proof}

\section{An Upper Bound for the Erd\H{o}s Sum} 
Our bounds on the Mertens product are particularly well-suited for bounding subsets of a primitive set $A$ whose members share a smallest irreducible common factor. Formally speaking, we choose an arbitrary ordering of $\I_q$ that respects increasing degree and define $p(f)$ and $P(f)$ to be the monic irreducible factor of $f$ which has least and greatest index according to this ordering, respectively. Then, we let $A'_p = \{a \in A: p(a) = p\}$ and note that $\{A_p'\}_{p \in \I_q}$ is a partition of $A$. Because the Erd\H{o}s sum converges for any primitive set $A$ \cite{us19}, we can obtain an upper bound for $\f(A)$ by summing together upper bounds for $\f(A'_p)$ over all monic irreducibles $p$. 

When $p \notin A$, we can bound $\f(A'_p)$ by adapting an argument that Lichtman and Pomerance \cite{lichtmanPom19} developed for the integer case. We let $g(a)$ represent the asymptotic density of monic multiples of $a$ all of whose factors have degree at least that of $P(a)$, whose formula is given by
\[g(a) \ = \  \frac{1}{||a||} \prod_{\substack{f \in \mathcal{I}_q \\ f < P(a)}} \Big( 1 - \frac{1}{||f||}\Big).\]
We also define $d(f) = \deg p(f)$ and $D(f) = \deg P(f)$. Then we have the following bound for $\f(A_p')$:

\begin{proposition}\label{prop.g} 
Let $A \subset \F_q[x]$ be primitive and $p \notin A$ be irreducible. Unless $q = 2$ and $\deg p = 1$, 
\[\f(A_p')  \ < \  e^{\gamma} g(p) .\]
\end{proposition}

\begin{proof}
For each $a \in A'_p$, Proposition \ref{prop.mertenslb} gives
\[g(a) \ = \  \frac{1}{||a||} \prod_{\substack{f \in \mathcal{I}_q \\ f < P(a)}} \Big( 1 - \frac{1}{||f||}\Big) \ > \  \frac{1}{||a||} \prod_{\substack{f \in \I_q \\ \deg f \leq D(a)}} \Big( 1 - \frac{1}{||f||}\Big) 
 \ \geq \  \frac{1}{e^\ga (D(a)+1) ||a||}.\]Note that this holds even in the case $q = 2$, since $\deg p > 1$ implies $D(a) > 1$. When $p \notin A$, we have $\deg a \geq D(a) + 1$, so
\[g(a) \ > \  \frac{1}{e^\gamma ||a|| \deg a } \ = \  \frac{1}{e^{\gamma}} \f(a).\] This gives us the preliminary upper bound 
\[\f(A_p') \ = \  \sum_{a \in A_p'} \f(a)  \ < \  e^\ga \sum_{a \in A_p'} g(a).\]

To bound this last summation, note that $A_p' \subset A$ is primitive. Thus if we define $S_a = \{fa: p(f) \geq P(a)\}$ for each $a \in A'_p$, we see that the $S_a$ must be pairwise disjoint. Because $S_a$ consists of the monic multiples of $a$ whose other irreducible factors have index at least $P(a)$, the asymptotic density of $S_a$ is $g(a)$. $S_a$ is contained in the set of all polynomials $f$ such that $p(f) = p(a) = p$, which has asymptotic density $g(p)$. Because the $S_a$ are disjoint,
\[\sum_{a \in A_p'} g(a)  \ \leq \  g(p).\]
It follows that $\f(A_p') \leq e^\ga g(p)$, as desired.
\end{proof}

When $\deg p = 1$, it is possible to obtain bounds for $\f(A'_p)$ that are tighter than those which would be obtained by applying Proposition \ref{prop.g} directly. In order to do so, we will partition each $A'_p$ into subsets $A^t$, which consist of elements of $A'_p$ that are exactly divisible by $t$. 

\begin{proposition}\label{prop.t}
Let $t$ be a product of degree $1$ irreducibles and let $A$ be primitive. Define $A^t = \{a \in A: t|a \textnormal{ and } d(a/t) \geq 2\}$. If $t \notin A$,
\[\f(A^t)  \ < \  \frac{e^\gamma}{||t||}\sum_{\substack{p \in \mathcal{I}_q\setminus A \\ \deg p > 1}} g(p).\]
\end{proposition}

\begin{proof}
Let $B^t = \{a/t : a \in A^t\}$. Note that $B^t$ is primitive, and furthermore that if $t \notin A$, then $B^t \neq \{1\}$\ignore{ and $\f((B^t)'_1) = \emptyset$???}. Because $t \notin A$, we have
\[\f(A^t) \ = \  \sum_{\substack{p \in \mathcal{I}_q \\ \deg p > 1}}  \f(t \cdot (B^t)'_p)\]
where set multiplication is defined in the natural way: $f\cdot S = \{f \cdot s : s \in S\}$. If $p \notin (B^t)'_p$ then Proposition \ref{prop.g} gives us the strict inequality in
\[\f(t \cdot (B^t)'_p)  \ \leq \   \frac{\f((B^t)'_p)}{||t||}  \ < \  \frac{e^\gamma g(p)}{||t||}.\]
If $p \in (B^t)'_p$ then $\f(t \cdot(B^t)'_p) = \f(tp)$. Using Proposition \ref{prop.mertenslb}, 
\[\f(tp) \ = \  \frac{1}{||tp|| \deg tp}  \ \leq \  \frac{1}{||tp||(\deg p + 1)}  \ < \  \frac{e^\gamma}{||tp||} \prod_{\substack{f \in \mathcal{I}_q \\ \deg f \leq \deg p}} \left(1 - \frac{1}{||f||} \right)  \ < \  \frac{e^\gamma g(p)}{||t||}.\]
Note that if $\f((B^t)'_p) \neq 0$ then $A$ contains nontrivial multiples of $p$ and thus $p \notin A$. It follows that
\[ \sum_{\substack{p \in \mathcal{I}_q \\ \deg p > 1}} \f(t \cdot (B^t)'_p)  \ < \  \frac{e^\gamma}{||t||}\sum_{\substack{p \in \mathcal{I}_q\setminus A \\ \deg p > 1}} g(p). \qedhere\]
\end{proof}

By summing over all possible $t$, we can establish an upper bound for $\sum_{p \in \mathcal{I}_q, \deg p = 1} \f(A'_p)$. While our initial bound will depend on the proportion of irreducibles not contained in $A$, we will later determine for which proportions this bound is maximized to obtain an upper bound independent of this quantity. Because the proof depends on Proposition \ref{prop.g}, which does not always apply when $q = 2$, we will first establish the result for $q \geq 3$ and then consider the case $q = 2$  separately. 

\begin{lemma}\label{lem.1}
Let $\alpha$ be the proportion of degree $1$ irreducibles not contained in $A$. When $q \geq 3$,
\begin{equation}\sum_{\substack{p \in \mathcal{I}_q \\ \deg p = 1}} \f(A'_p)  \ < \  (1-\alpha) + e^\gamma \sum_{\substack{p \in \mathcal{I}_q\setminus A \\ \deg p > 1}} g(p) + e^\gamma \left(1- \frac{1}{q}\right)^q \left( \left(1- \frac{1}{q} \right)^{-\alpha q} - \alpha - 1 \right).\label{eq:lembd} \end{equation}
\end{lemma}
\begin{proof}
Let $K$ be the set of degree 1 irreducibles not contained in $A$, and note that $|K| = \alpha q$. Let $t$ denote a product of degree 1 irreducibles as before. If $t$ is a multiple of a degree $1$ irreducible contained in $A$, then $A^t = \emptyset$ and so $\f(A^t) = 0$. Hence
\[\sum_{\substack{p \in \mathcal{I}_q \\ \deg p = 1}} \f(A'_p) \ = \  \sum_{\substack{t \in A\\ \deg t = 1}} \f(t) +
\sum_{\substack{t \notin A\\ \deg t = 1}} \f(A^t) +
\sum_{\substack{t \notin A, \mathcal{P}(t) \subset K\\ \deg t > 1}} \f(A^t) +
\sum_{\substack{t \in A, \mathcal{P}(t) \subset K\\ \deg t > 1}} \f(t),\]
where $\mathcal{P}(t)$ denotes the set of irreducible factors of $t$. The first sum equals 
\[\sum_{\substack{t \in A\\ \deg t = 1}} \frac{1}{q} \ = \  \frac{(1-\alpha)q}{q} = 1-\alpha,\]
and we can use Proposition \ref{prop.t} to bound the second sum as
\[\sum_{\substack{t \notin A\\ \deg t = 1}} \f(A^t)  \ < \  \sum_{\substack{t \notin A \\ \deg t = 1}} \frac{e^\gamma}{||t||}\sum_{\substack{p \in \mathcal{I}_q\setminus A \\ \deg p > 1}} g(p) \ = \  \alpha e^\gamma \sum_{\substack{p \in \mathcal{I}_q\setminus A \\ \deg p > 1}} g(p).\]
Similarly, we can bound the third sum as
\[\sum_{\substack{t \notin A, \mathcal{P}(t) \subset K\\ \deg t > 1}} \f(A^t)  \ < \  \sum_{\substack{t \notin A,\mathcal{P}(t) \subset K\\ \deg t > 1}} \frac{e^\gamma}{||t||}\sum_{\substack{p \in \mathcal{I}_q\setminus A \\ \deg p > 1}} g(p)  \ \leq \  \sum_{\substack{t \notin A,\mathcal{P}(t) \subset K\\ \deg t > 1}} \frac{e^\gamma}{||t||}\sum_{\substack{p \in \mathcal{I}_q\\ \deg p > 1}} g(p).\]
Finally, when $q\geq 3$ and $\deg t > 1$, we can deduce that
\begin{equation}
    \frac{1}{||t||\deg t}  \ \leq \  \frac{1}{2||t||}  \ \leq \  \frac{e^\gamma}{||t||} \left(1 - \frac{1}{q} \right)^q, \label{eq:degreebd}
\end{equation} 
where the last inequality holds because $\left(1-\frac{1}{q}\right)^q$ is increasing with $q$ and  $e^\ga \left(1 - \frac{1}{3} \right)^3 = 0.52772\ldots > \frac{1}{2}$. This allows us to bound the final sum as
\[ 
\sum_{\substack{t \in A, \mathcal{P}(t) \subset K\\ \deg t > 1}} \f(t)  \ \leq \  \sum_{\substack{t \in A, \mathcal{P}(t) \subset K\\ \deg t > 1}} \frac{e^\gamma}{||t||} \left(1 - \frac{1}{q} \right)^q,\]
upon which summing together our four bounds gives 
\[\sum_{\substack{p \in \mathcal{I}_q \\ \deg p = 1}} \f(A'_p)  \ < \  (1-\alpha) +\alpha e^\gamma \sum_{\substack{p \in \mathcal{I}_q\setminus A \\ \deg p > 1}} g(p) + \sum_{\substack{t \notin A,\mathcal{P}(t) \subset K\\ \deg t > 1}} \frac{e^\gamma}{||t||}\sum_{\substack{p \in \mathcal{I}_q\\ \deg p > 1}} g(p) + \sum_{\substack{t \in A, \mathcal{P}(t) \subset K\\ \deg t > 1}} \frac{e^\gamma}{||t||} \left(1 - \frac{1}{q} \right)^q.\]
To simplify the third term of this expression, note that
\[\sum_{\substack{p \in \mathcal{I}_q \\ \deg p = 1}} g(p) \ = \  \frac{1}{q}\sum_{i=0}^{q-1} \left(1 - \frac{1}{q} \right)^i \ = \  \frac{1}{q}\left(\frac{\left( 1 - \frac{1}{q} \right)^{q}-1}{\left( 1 - \frac{1}{q} \right) - 1} \right) \ = \  1 - \left(1 - \frac{1}{q}\right)^q.\]
Because every polynomial is divisible by an irreducible, $\sum_{p \in \mathcal{I}_q} g(p) = 1$, which means 
\[\sum_{\substack{p \in \mathcal{I}_q \\ \deg p > 1}} g(p) \ = \  1 - \sum_{\substack{p \in \mathcal{I}_q \\ \deg p = 1}} g(p) \ = \  \left(1 - \frac{1}{q}\right)^q.\]
Using this formula for $\sum_{\substack{p \in \mathcal{I}_q \\ \deg p > 1}}  g(p)$ and combining the last two terms of the expression, our bound becomes
\[ \sum_{\substack{p \in \mathcal{I}_q \\ \deg p = 1}} \f(A'_p) \ \leq \  (1-\alpha) +\alpha e^\gamma \sum_{\substack{p \in \mathcal{I}_q\setminus A \\ \deg p > 1}} g(p) + \sum_{\substack{\mathcal{P}(t) \subset K\\ \deg t > 1}} \frac{e^\gamma}{||t||} \left( 1 - \frac{1}{q}\right)^q.\]
We can evaluate the sum of reciprocals in the last term using an Euler-like product expansion because $\F_q[x]$ is a unique factorization domain (so that every $t$ is a product of exactly one combination of elements of $K$):
\[\sum_{\substack{\deg t > 1 \\ \mathcal{P}(t) \subset K}} \frac{1}{||t||} \ = \  \left(1 + \frac{1}{q} + \frac{1}{q^2} + \ldots \right)^{\alpha q} - \frac{\alpha q}{q} - 1 \ = \  \left(1 - \frac{1}{q} \right)^{-\alpha q} -\alpha  - 1. \]
Substituting this into our inequality above gives the desired bound.
\end{proof}

\begin{proposition}\label{prop.1}
 When $q \geq 3$,
\[\sum_{\substack{p \in \mathcal{I}_q \\ \deg p = 1}} \f(A'_p)  \ \leq \  \max\bigg\{ 1, e^\gamma \sum_{\substack{p \in \mathcal{I}_q\setminus A \\ \deg p > 1}} g(p) + e^\gamma \left( 1 - 2 \left(1 - \frac{1}{q} \right)^q \right)\bigg\} .\]
\end{proposition}
\begin{proof}
We will show that the upper bound in Lemma \ref{lem.1} is maximized on $[0, 1]$ at $\alpha = 0$ or $\alpha = 1$. Fix the irreducible polynomials of degree at least 2 contained in $A$ and define
\[C \ := \ \sum_{\substack{p \in \mathcal{I}_q\setminus A \\ \deg p > 1}} g(p).\]
Recall from Lemma \ref{lem.1} that \[C \ = \  \sum_{\substack{p \in \mathcal{I}_q\setminus A \\ \deg p > 1}} g(p)  \ \leq \  \sum_{\substack{p \in \mathcal{I}_q \\ \deg p > 1}} g(p) \ = \  \left(1 - \frac{1}{q}\right)^q.\]  Thus $C$ is independent of $\alpha$ and bounded between $0$ and $\frac{1}{e}$. Hence, we can treat $C$ as a constant and prove the claim for all possible values of $C$. Taking the second derivative of 
\[(1-\alpha) + C\alpha e^\gamma+ e^\gamma \left(1- \frac{1}{q}\right)^q \left( \left(1- \frac{1}{q} \right)^{-\alpha q} -\alpha  - 1 \right),\]
the upper bound in \eqref{eq:lembd}, with respect to $\alpha$ gives 
\[e^\gamma \left(1- \frac{1}{q} \right)^q \log \left( \left( 1 - \frac{1}{q}\right)^{-q} \right)^2 \left(1 - \frac{1}{q} \right)^{-\alpha q}.\]
This quantity is always positive, so our upper bound is maximized at an endpoint of $[0,1]$. When $\alpha = 0$, it equals 1, and when $\alpha = 1$, it equals
\[ Ce^\gamma + e^\gamma \left(1- \frac{1}{q}\right)^q\left( \left(1- \frac{1}{q} \right)^{-q} - 2 \right),\]
so $\sum_{\substack{p \in \mathcal{I}_q, \deg p = 1}} \f(A'_p)$ is bounded above by the greater of these two values.
\end{proof}

By adding the upper bounds for $\f(A'_p)$ with $\deg p > 1$ to the upper bound obtained in Proposition \ref{prop.1}, we arrive at a final upper bound for our Erd\H{o}s sum when $q \geq 3$.

\begin{theorem} \label{thm:univbd}
For $3\leq q \leq 19$ we have $\f(A)\leq e^\gamma=1.78107\ldots$ and for $q>19$ we have
\[\f(A)   \ \leq \  1 + e^\gamma \left(1-\frac{1}{q}\right)^q + \sum_{\substack{p \in \mathcal{I}_q \\ \deg p > 1}}\frac{1}{||p||(\deg p +1)(\deg p )}  \ < \  1+e^{\gamma-1}+\frac{\pi^2-9}{6} \ = \  1.80015\ldots.\]
\end{theorem}

\begin{proof}
We can bound $\f(A)$ by summing over disjoint subsets of $A$, just as we did in the previous proposition:
\begin{align}
    \f(A) \ &= \ \sum_{\substack{p \in \mathcal{I}_q \\ \deg p = 1}} \f(A'_p) + \sum_{\substack{p \in \mathcal{I}_q\setminus A \\ \deg p > 1}} \f(A'_p) + \sum_{\substack{p \in \mathcal{I}_q\cap A \\ \deg p > 1}} \f(p). \nonumber \\
    & < \ \sum_{\substack{p \in \mathcal{I}_q \\ \deg p = 1}} \f(A'_p) + e^\gamma \sum_{\substack{p \in \mathcal{I}_q\setminus A \\ \deg p > 1}} g(p) + \sum_{\substack{p \in \mathcal{I}_q\cap A \\ \deg p > 1}} \f(p). \label{eq:eulerbound}
\end{align}
Here, we have used Proposition \ref{prop.g} for the second term.  If the first sum in the bound is greater than 1, then we can bound it using Proposition \ref{prop.1} as
\[\sum_{\substack{p \in \mathcal{I}_q \\ \deg p = 1}} \f(A'_p)  \ \leq \  e^\gamma \sum_{\substack{p \in \mathcal{I}_q\setminus A \\ \deg p > 1}} g(p) + e^\gamma \left( 1 - 2 \left(1 - \frac{1}{q} \right)^q \right).\]
We then bound the third sum using Proposition \ref{prop.mertenslb} as  
\begin{align*}
    \sum_{\substack{p \in \mathcal{I}_q\cap A \\ \deg p > 1}} \f(p) \ = \  \sum_{\substack{p \in \mathcal{I}_q\cap A \\ \deg p > 1}} \frac{1}{||p||\deg p} \ &\leq \  \sum_{\substack{p \in \mathcal{I}_q\cap A \\ \deg p > 1}}\frac{1}{||p||(\deg p+1)}\\
& < \ 2 \sum_{\substack{p \in \mathcal{I}_q\cap A \\ \deg p > 1}} \frac{e^\gamma}{||p||} \prod_{\substack{f \in \mathcal{I}_q \\ \deg f \leq \deg p}} \left( 1- \frac{1}{||f||} \right) \\
&< \ 2 e^\gamma \sum_{\substack{p \in \mathcal{I}_q\cap A \\ \deg p > 1}} g(p) .
\end{align*}
Now, inserting both of these estimates and combining all three sums of \eqref{eq:eulerbound},
\begin{align*}
    \f(A) \ & < \ e^\gamma \left(1 - 2 \left(1 - \frac{1}{q} \right)^q \right) + 2e^\gamma \sum_{\substack{p \in \mathcal{I}_q \\ \deg p > 1}} g(p) \\
    &\leq \ e^\gamma \left(1 - 2 \left(1 - \frac{1}{q} \right)^q \right) + 2e^\gamma \left(1 - \frac{1}{q} \right)^q \ = \  e^\gamma .
\end{align*}

If instead, the first sum of \eqref{eq:eulerbound} is at most 1, then we treat the third sum in that expression more delicately.  Since $\deg p > 1$, we can bound the terms of this last sum with Proposition \ref{prop.mertenslb}:
\begin{align*}\f(p) \ = \  \frac{1}{||p||\deg p} \ &\leq \ \frac{1}{||p||(\deg p+1)}\cdot\frac{\deg p +1}{\deg p} \\
&< \ \frac{e^\gamma}{||p||} \prod_{\substack{f \in \mathcal{I}_q \\ \deg f \leq \deg p}} \left( 1- \frac{1}{||f||} \right)+\frac{1}{||p||(\deg p +1)(\deg p )} \\
&< \ e^\gamma g(p) +\frac{1}{||p||(\deg p +1)(\deg p )}.
\end{align*}
Inserting this bound in the third sum of \eqref{eq:eulerbound}, combining it with the second sum, and applying Proposition \ref{prop.1} to the first sum, we have \begin{align}
    \f(A)  &< \ 1 + e^\gamma \sum_{\substack{p \in \mathcal{I}_q \\ \deg p > 1}} g(p) + \sum_{\substack{p \in \mathcal{I}_q \\ \deg p > 1}}\frac{1}{||p||(\deg p +1)(\deg p )}\nonumber \\
    &\leq \ 1 + e^\gamma \left(1-\frac{1}{q}\right)^q + \sum_{\substack{p \in \mathcal{I}_q \\ \deg p > 1}}\frac{1}{||p||(\deg p +1)(\deg p )}. \label{eq:fAless1bd}
\end{align} 
By numerically computing the sum over irreducible polynomials, we find that this bound is less than $e^\gamma$ for all $q\leq 19$. To obtain a bound independent of $q$, we can bound the third sum above as \
\begin{align*}
    \sum_{\substack{p \in \mathcal{I}_q \\ \deg p > 1}}\frac{1}{||p||(\deg p +1)(\deg p )}\ = \ \sum_{n=2}^\infty \frac{\pi'_q(n)}{n(n+1)q^n} \ < \ \sum_{n=2}^\infty \frac{1}{n^2(n+1)} \ = \  \frac{\pi^2-9}{6}
\end{align*}
and note that $\left(1-\frac{1}{q}\right)^q$ increases with $q$ to $\frac{1}{e}$, which gives  \[\f(A)  \ < \  1+e^{\gamma-1}+\frac{\pi^2-9}{6} \ = \  1.800153\ldots. \qedhere\]

\end{proof}

Our proof above fails when $q = 2$ because the inequality \eqref{eq:degreebd} in Lemma \ref{lem.1} no longer holds. However, it becomes true if we introduce a correction factor of $2/e^{\ga} = 1.122918\ldots$:
\[\frac{1}{||t||\deg t}  \ \leq \  \frac{1}{2||t||} \ = \ \frac{2}{e^\gamma}\cdot \frac{e^\gamma}{||t||}\left(1 - \frac{1}{2} \right)^2.\]
By modifying the propositions above to take this correction factor into account, we can obtain an upper bound for the Erd\H{o}s sum in the case $q = 2$. 

\begin{theorem}
When $q = 2$, 
\[ \f(A)  \ < \  1+\frac{e^\ga}{2} \ = \  1.890536\ldots .\]
\end{theorem}
\begin{proof}
Because we are introducing a factor of $2/e^\ga$ in the right hand side of 
\[\frac{1}{||t||\deg t}  \ \leq \  \frac{e^\gamma}{||t||}\left(1 - \frac{1}{q} \right)^q,\]
our upper bound from Lemma \ref{lem.1} becomes 
\[\sum_{\substack{p \in \mathcal{I}_q \\ \deg p = 1}} \f(A'_p)  \ \leq \  (1-\alpha) +\alpha e^\gamma \sum_{\substack{p \in \mathcal{I}_q\setminus A \\ \deg p > 1}} g(p) + 2 \left(1- \frac{1}{q}\right)^q \left( \left(1- \frac{1}{q} \right)^{-\alpha q} -\alpha  - 1 \right).\]
The right hand expression is still convex as a function of $\alpha$ by the same reasoning as Proposition \ref{prop.1}, so 
\begin{equation} \sum_{\substack{p \in \mathcal{I}_2 \\ \deg p = 1}} \f(A'_p)  \ \leq \  \max\bigg\{ 1, e^\gamma \sum_{\substack{p \in \mathcal{I}_2\setminus A \\ \deg p > 1}} g(p) + 2 \left( 1 - 2 \left(1 - \frac{1}{2} \right)^2 \right)\bigg\}.\label{eq:2maxbd} \end{equation}
If $\sum_{\substack{p \in \mathcal{I}_q \\ \deg p = 1}} \f(A'_p) \leq 1$, then the same argument as in Theorem \ref{thm:univbd} when this sum is bounded by 1 applies, and we can bound
\[     \f(A)  \ \leq \  1 + e^\gamma \left(1-\frac{1}{2}\right)^2 + \sum_{\substack{p \in \mathcal{I}_2 \\ \deg p > 1}}\frac{1}{||p||(\deg p +1)(\deg p )}  \ < \  e^\gamma \]
as in \eqref{eq:fAless1bd}.
Otherwise, inserting the second bound in \eqref{eq:2maxbd} into \eqref{eq:eulerbound} and following as above yields
 \[\f(A)  \ < \  2 \left(1 - 2 \left(1 - \frac{1}{2} \right)^2 \right) + 2e^\gamma \left(1 - \frac{1}{2} \right)^2 \ = \  1 + \frac{e^\ga}{2}. \qedhere\]
\end{proof}
\section{The Banks-Martin Inequality}
Recall that the analogue of the Banks-Martin conjecture for $\F_q[x]$ states that 
\[\f(\mathcal{I}_{1,q})  \ > \  \f(\mathcal{I}_{2,q})  \ > \  \f(\mathcal{I}_{3,q}) \ldots, \]
where $\mathcal{I}_{k,q}$ represents the set of monic polynomials in $\mathbb{F}_q[x]$ with $k$ irreducible factors counted with multiplicity. Analogously to the observation of Lichtman \cite{lichtman}, we find that the conjecture is false for $q=2$, $3$, and $4$ by direct numerical computation in Section \ref{sec:numresults}.  However, in this section we will show that for each $k$, there exists $q_k$ such that the inequality holds up to $\f(\mathcal{I}_{k,q})$ for all $q \geq q_k$, and we will establish an upper bound on the size of $q_k$. 

\subsection{Bounds for $\pi'_{q,k}(n)$ and $\pi^*_k(n)$} 
Let $\pi'_k(n)$ denote the number of monic polynomials of degree $n$ in $\F_q[x]$ with $k$ irreducible divisors including multiplicity. Since 
\[\f(\mathcal{I}_{k,q}) \ = \ \sum_{a \in \mathcal{I}_{k,q}} \frac{1}{||a|| \deg a} \ = \  \sum_{n=1}^{\infty} \frac{\pi'_k(n)}{nq^n},\]
the growth of $\f(\mathcal{I}_{k,q})$ is determined by $\pi'_k(n)$. Similarly, if we let $\mathcal{I}^*_{k,q} = \{f \in \mathcal{I}_{k,q} : f \textnormal{ squarefree}\}$ and $\pi^*_k(n)$ denote the number of squarefree monics of degree $n$ with $k$ irreducible divisors, then the growth of $\f(\mathcal{I}^*_{k,q})$ is determined by $\pi^*_k(n)$. Thus we can obtain bounds for $\f(\mathcal{I}_{k,q})$ and $f(\mathcal{I}^*_{k,q})$ by bounding their respective counting functions $\pi'_k(n)$ and $\pi^*_k(n)$.

\ignore{\begin{lemma}\label{lem.piasymp}
\[k \pi^*_k(n)  \ \leq \  \sum_{\mathclap{\substack{p \in \I \\ \deg p \leq n}}} \pi^*_{k-1}(n-\deg p).\]
\end{lemma}
 
\begin{proof}
Let $P_{k,n}$  be the set of monic squarefree polynomials of degree $n$ with exactly $k$ irreducible factors. Consider the set of ordered pairs $A_{k,n} = \{(f, p) \in P_{k,n} \times \I \textnormal{ s.t. } p|f\}$. As each polynomial in $P$ has exactly $k$ irreducible monics that divide it, each of these irreducibles is distinct (because $f$ is squarefree), and all of these polynomials have degree less than $n$, we have that \[|A_{k,n}| \ = \  k|P_{k,n}| \ = \  k\pi^*_k(n).\]

On the other hand, each monic irreducible $p \in \I$ with degree at most $n-1$ has at most $\Pi_{k-1}(n-\deg p)$ polynomials $f \in P_{k-1,n-\deg p}$ which can multiply with it to get a polynomial in $P_{k,n}$, so
\[k\pi^*_k(n) \ = \  |A_{k,n}|  \ \leq \  \sum_{\mathclap{\substack{p \in \I \\ \deg p < n}}} \pi^*_{k}(n-\deg p). \]
\end{proof}
}

\begin{proposition}\label{prop.piasympgeneral}
\[\pi^*_k(n)  \ \leq \  \frac{1}{k!} \sum_{\substack{j_1, \ldots, j_k \\ j_1 + \ldots + j_k = n}} \pi'_q(j_1)\pi'_q(j_2)\ldots\pi'_q(j_k).\]
\end{proposition}

\begin{proof}
\ignore{We proceed by induction on $k$. The base case $k = 1$ is immediate, and if we assume that the claim is true for $k = l$, then 
\[\pi^*_{l+1}(n)  \ \leq \  \frac{1}{l+1} \sum_{\mathclap{\substack{p \in \I \\ \deg p \leq n}}} \pi^*_{l}(n-\deg p) \ = \  \frac{1}{l+1} \sum_{j_{l+1} \ = \  1}^{n - 1} \pi'_q(j_{l+1}) \pi^*_{l}(n - j_{l+1}). \]
By Lemma \ref{lem.piasymp},
\[ \frac{1}{l+1} \sum_{j_{l+1} = 1}^{n - 1} \pi'_q(j_{l+1}) \pi^*_{l}(n - j_{l+1})  \ \leq \  \frac{1}{(l+1)!} \sum_{\substack{j_1, \ldots, \j_{l+1} \\ j_1 + \ldots + j_{l+1} = n}} \pi'_q(j_1)\pi'_q(j_2)\ldots\pi'_q(j_{l+1}).
\]
which completes the induction.}
 $\pi^*_k$ counts polynomials of the form $p_1p_2\ldots p_k$, where the $p_i$ are distinct irreducibles with degrees $j_i$ such that $j_1 + \ldots + j_k = n$. There are $\pi'_q(j_1)\pi'_q(j_2)\ldots \pi'_q(j_k)$ ways to choose $k$ irreducibles with respective degrees $j_1, \ldots j_k$. However, this product includes in its count some non-squarefree polynomials, and for the polynomials that are squarefree, there are $k!$ different ways we can order them to obtain the same product. Hence, summing over all tuples $j_1, \ldots j_k$ that sum to $n$ and dividing by $k!$ gives an upper bound for $\pi^*_k(n)$.
\end{proof}

\begin{proposition}\label{prop.pi_k bound}
\[\pi'_k(n)  \ \geq \  \frac{1}{k!} \sum_{\substack{j_1, \ldots, j_k \\ j_1 + \ldots + j_k = n}} \pi'_q(j_1)\pi'_q(j_2)\ldots\pi'_q(j_k).\]
\end{proposition}
\begin{proof}
We can write each polynomial counted by $\pi'_k(n)$ as $p_1p_2\ldots p_k$, where the $p_i$ are not necessarily distnct irreducibles. There are $\pi'_q(j_1)\pi'_q(j_2)\ldots\pi'_q(j_k)$ ways to choose the irreducibles $p_1, \ldots, p_k$ with degrees $j_1, \ldots j_k$, and these irreducibles can be ordered in $k!$ ways. However, if not all irreducibles are distinct, then some reorderings will result in the same polynomial, meaning that dividing by $k!$ undercounts the total number of polynomials. Summing over tuples $j_1, \ldots j_k$ that sum to $n$ gives a lower bound for $\pi'_k(n)$.
\end{proof}

\subsection{An Upper Bound for $\f(\mathcal{I}_{k,q})$}
In order to obtain an upper bound for $\f(\mathcal{I}_{k,q})$, we will first bound $\f(\mathcal{I}^*_{k,q})$ using our bounds for $\pi_k^*(n)$.  We will use the following result of Mordell, stated in greater generality than is needed here, as it will be useful for the computations in Section \ref{sec:computation}.

\begin{theorem}[Mordell \cite{MR0100181}] \label{thm:mordell}
For any positive integer $k$ and $a>-k$ we have \begin{align}
    \sum_{1\leq n_1,n_2,\ldots, n_k} \frac{1}{n_1n_2\cdots n_k(n_1+n_2+ \cdots + n_k +a)} \ &= \ k!\left(1+\frac{1-a}{1!2^{k+1}} + \frac{(1-a)(2-a)}{2!3^{k+1}}+\cdots\right) \label{eq:mordellsum}    \\
    & = \ k! \sum_{i=0}^\infty \frac{\left(-1\right)^i}{(i+1)^{k+1}} \binom{a-1}{i}, \nonumber
\end{align}
where the sum ranges over all $k$-tuples of positive integers.
\end{theorem}
\begin{remark}Note that when $a=0$, the expression on the right in \eqref{eq:mordellsum} is equal to $k!\zeta(k+1)$, where  $\zeta(s) = \sum_{n = 1}^{\infty} 1/n^s$ is the Riemann zeta function. When $a$ is a positive integer, the right hand sum is finite, so the result is a rational number.
\end{remark}

\begin{proposition}\label{prop.mordell}
\[\f(\mathcal{I}^*_{k,q})  \ \leq \  \zeta(k+1).\]
\end{proposition}

\begin{proof}
By Proposition \ref{prop.piasympgeneral}, we have 
\[\pi^*_k(n)  \ \leq \  \frac{1}{k!} \sum_{\substack{j_1, \ldots, j_k \\ j_1 + \ldots + j_k = n}} \pi'_q(j_1)\pi'_q(j_2)...\pi'_q(j_k).\]
Recalling that $\pi_q'(n) \leq \frac{q^n}{n}$, we see that
\[\f(\mathcal{I}^*_{k,q}) \ = \  \sum_{n=1}^{\infty} \frac{\pi^*_k(n)}{nq^n} 
 \ \leq \  \frac{1}{k!} \sum_{n = 1}^{\infty} \sum_{\substack{j_1 \ldots j_k \\ j_1 + \ldots + j_k = n}} \frac{\pi'_q(j_1)...\pi'_q(j_k)}{nq^n}
 \ \leq \  \frac{1}{k!} \sum_{n = 1}^{\infty} \sum_{\substack{j_1 \ldots j_k \\ j_1 + \ldots + j_k = n}} \frac{1}{nj_1j_2 \ldots j_k}.\]
By Theorem \ref{thm:mordell}, this right hand sum equals $\zeta(k + 1)$, giving us the desired result. 
\end{proof}

\begin{proposition} \label{upper bound for f(I_k)}
For $k \geq 3$,
\[\f(\mathcal{I}_{k,q})  \ \leq \  \zeta(k+1) + \log \left(\frac{q}{q-1}\right)\zeta(k-1). \]
\end{proposition}
\begin{proof}
Proposition \ref{prop.mordell} gives us an upper bound for the Erd\H{o}s sum over squarefree elements of $\mathcal{I}_{k,q}$; all that remains is to consider the contribution from the non-squarefree terms, which are polynomials of the form $p_1^2p_2\ldots p_{k-1}$, where the $p_i$ are not necessarily distinct irreducibles. 
If we let $j_i = \deg p_i$, then there are at most $\pi'_q(j_1)\pi'_q(j_2)\ldots\pi'_q(j_{k-1})$ polynomials $p_1^2p_2\ldots p_{k-1}$ whose factors have the corresponding degrees $j_1, \ldots j_{k-1}$. Hence our sum over non-squarefree terms is bounded above by

\begin{align*}
\sum_{j_1 \ldots j_{k-1}} \frac{\pi'_q(j_1)\ldots \pi'_q(j_{k-1})}{(2j_1 + j_2+ \ldots + j_{k-1})q^{2j_1 + j_2 +  \ldots + j_{k-1}}}
\ &\leq \ \sum_{j_1 \ldots j_{k-1}} \frac{1}{j_1j_2\ldots j_{k-1}(j_1 + \ldots + j_{k-1})q^{j_1}}\\
\ &\leq \ \sum_{j_1}\frac{1}{j_1q^{j_1}}\sum_{j_2 \ldots j_{k-1}} \frac{1}{j_2\ldots j_{k-1}(j_2 + \ldots + j_{k-1})}\\
&= \ \log \left(\frac{q}{q-1}\right)\zeta(k-1),
\end{align*}
where once again we have used Theorem \ref{thm:mordell} to obtain the last equality. Note that the second line requires $k \geq 3$. Combining our bounds for the squarefree and non-squarefree elements in $\mathcal{I}_{k,q}$ yields a total upper bound of
\[ \zeta(k+1) + \log \left(\frac{q}{q-1}\right)\zeta(k-1). \vspace{-1em}\]
\end{proof}

Since $\log(\frac{q}{q-1})$ decreases to zero as $q$ tends to infinity, this upper bound gets arbitrarily close to $\zeta(k+1)$ whenever $\zeta(k-1)$ converges. However, we will need a separate bound for $k = 2$.

\begin{proposition} \label{upper bound for f(I_2)}
\[\f(\mathcal{I}_{2,q})  \ \leq \  \zeta(3) + \frac{1}{2}\textnormal{Li}_2\left(\frac{1}{q}\right),\]
where $\textnormal{Li}_2(x)=\sum_{k=1}^{\infty}\frac{x^k}{k^2}$ is the dilogarithm function.
\end{proposition}

\begin{proof}
The only elements in $\mathcal{I}_{2,q}$ that are not squarefree are the squares of irreducibles, which contribute 
\[ \sum_{f \in \mathcal{I}_q} \frac{1}{||f^2|| \deg (f^2)} \ = \  \sum_{n = 1}^{\infty} \frac{\pi'_q(n)}{2nq^{2n}}  \ \leq \ 
\sum_{n = 1}^{\infty} \frac{1}{2n^2q^n} \ = \   \frac{1}{2}\textnormal{Li}_2\left(\frac{1}{q}\right)\]
to the upper bound. Since the Erd\H{o}s sum over the squarefrees is bounded above by $\zeta(3)$ by Proposition \ref{prop.mordell},
\[\f(\mathcal{I}_{2,q}) \ = \  \f(\I^*_{2,q}) + f(\I_{2,q} \setminus \I^*_{2,q})  \ \leq \  \zeta(3) + \frac{1}{2}\textnormal{Li}_2\left(\frac{1}{q}\right). \vspace{-1em}\]
\end{proof}


\subsection{A Lower Bound for $\f(\mathcal{I}_{k,q})$}
Having established an upper bound for $\f(\mathcal{I}_{k,q})$, we now find a lower bound for the same sum.

\begin{proposition}\label{lower bound for f(I_k)}
$$ \f(\mathcal{I}_{k,q})  \ \geq \  
\bigg(1-\frac{\sqrt{q}}{q-1}\bigg)^k \ze(k+1).$$
\end{proposition}

\begin{proof}
By Proposition \ref{prop.pi_k bound}, we have
$$\f(\mathcal{I}_{k,q})  \ \geq \  \frac{1}{k!} \sum_{n=1}^{\infty} \sum_{j_1 + ... + j_k = n} \frac{\pi'_q(j_1)...\pi'_q(j_k)}{n q^n}.$$
Expanding this out using the lower bound from Proposition \ref{lower bound of pi'} gives
$$\f(\mathcal{I}_{k,q})  \ \geq \  \frac{1}{k!} \sum_{n=1}^{\infty} \sum_{j_1 + ... + j_k = n} \frac{(q^{j_1} - \frac{q}{q-1}\cdot q^{j_1/2}) ...(q^{j_k} - \frac{q}{q-1} \cdot q^{j_k/2} )}{n j_1 ... j_k q^n}.$$
Factoring out $q^j$ from each term in the numerator, the right hand sum becomes
$$\f(\mathcal{I}_{k,q})  \ \geq \  \frac{1}{k!} \sum_{n=1}^{\infty} \sum_{j_1 + ... + j_k = n} \frac{(1 - \frac{q}{q-1}\cdot q^{\m j_1/2}) ...(1 - \frac{q}{q-1} \cdot q^{\m j_k/2} )}{n j_1 ... j_k}.$$
As $j_i \geq 1$ for all $i$, we can bound the sum below by 
\begin{align*}
    \frac{1}{k!} \sum_{n=1}^{\infty} \sum_{j_1 + ... + j_k = n} \frac{(1 - \frac{q}{q-1}\cdot q^{\m 1/2})^k}{n j_1 ... j_k} \ &= \ \bigg(1 - \frac{\sqrt{q}}{q-1}\bigg)^k \frac{1}{k!} \sum_{n=1}^{\infty} \sum_{j_1 + ... + j_k = n} \frac{1}{n j_1 ... j_k} \\
    &= \ \bigg(1 - \frac{\sqrt{q}}{q-1}\bigg)^k \ze(k+1)
\end{align*}
where the last step uses Theorem \ref{thm:mordell}. \qedhere

\end{proof}

\subsection{The Banks-Martin Inequality for Fixed $k$}
The upper and lower bounds for $\f(\mathcal{I}_{k,q})$ established in Propositions \ref{upper bound for f(I_k)} and \ref{lower bound for f(I_k)} both approach $\zeta(k+1)$ as $q$ increases. Because $\zeta(n) > \zeta(n+1)$, it follows that for each $k \in \mathbb{N}$, there exists $q_k$ such that the following chain of inequalities holds in $\mathbb{F}_q[x]$ for all $q \geq q_k$:
\[\f(\mathcal{I}_{1,q})  \ > \  \f(\mathcal{I}_{2,q})  \ > \  \ldots  \ > \  \f(\mathcal{I}_{k,q}).\]

In the following theorem, we establish how large $q_k$ must be in order to guarantee that this chain of inequalities will hold.

\begin{theorem} \label{upper bound for q_k}
For each $k \in \N$, there exists an integer $q_k = O \big( k^2 4^k \big)$ such that 
\[\f(\mathcal{I}_{1,q})  \ > \  \f(\mathcal{I}_{2,q})  \ > \  \ldots  \ > \  \f(\mathcal{I}_{k,q})  \]
for all $q \geq q_k$. In particular, we have that $q_k < 4.03 \ \!(k-1)^2 \ \!4^k \ \!\zeta(k)^2$.
\end{theorem}

\begin{proof}
We first address the case $k = 2$, which must be handled separately because the bounds in Propositions \ref{upper bound for f(I_k)} and \ref{lower bound for f(I_k)} do not apply. Instead, we can use Propositions \ref{prop.recipsum} and \ref{upper bound for f(I_2)}, which, along with the fact that
\[\frac{\pi^2}{6} - \frac{q}{q-1} \textnormal{Li}_2\bigg(\frac{1}{\sqrt{q}}\bigg)  \ > \ 
\zeta(3) + \frac{1}{2} \textnormal{Li}_2\left(\frac{1}{q}\right)\]
for all $q \geq 11$,  are sufficient to prove that $\f(\I_{1, q}) \geq \f(\I_{2, q})$ for all such $q$. 

Now fix a natural number $k \geq 3$. In order to find a value of $q$ such that $\f(\I_{k'-1, q}) > \f(\I_{k', q})$ for all $k' \leq k$, it suffices to find $q$ such that 
\[\bigg(1-\frac{\sqrt{q}}{q-1}\bigg)^{k'-1} \ze(k')  \ > \  \zeta(k'+1) + \log \Big( \frac{q}{q-1} \Big) \zeta(k'-1)\]
for all $k' \leq k$ because of Propositions \ref{upper bound for f(I_k)} and \ref{lower bound for f(I_k)}. Equivalently, we need
\[\bigg(1-\frac{\sqrt{q}}{q-1}\bigg)^{k'-1} \ze(k') - \zeta(k'+1) - \log \Big( \frac{q}{q-1} \Big) \zeta(k'-1)  \ > \  0.\]
It is not too difficult to show that the left hand expression is increasing in $q$, which implies that if the inequality holds for $q_k$ then it holds for all $q > q_k$. It is more challenging to show that the expression is decreasing with respect to $k'$, but this can be accomplished for $q \geq 7$ and $k' \geq 4$ using bounds on the forward difference of the Riemann zeta function, as found in \cite{MR3768303}. Hence if $q \geq 7$ satisfies the inequality when $k' = k$, it will satisfy the inequality for all $4 \leq k' \leq k$. In fact, because the inequality holds when $k = 3$ and $q = 413$, and we have already shown that $\f(\I_{1, q}) \geq \f(\I_{2, q})$ for all $q \geq 11$, the inequality holding for $k' = k$ implies that it holds for all $k' < k$ as long as $q \geq 413$.

Thus our task is as follows: given any $k \geq 4$, we must find a value of $q$ such that
\[\bigg(1-\frac{\sqrt{q}}{q-1}\bigg)^{k-1} \ze(k) - \zeta(k+1) - \log \Big( \frac{q}{q-1} \Big) \zeta(k-1)  \ > \  0.\]
Bounding the coefficient of $\zeta(k)$ above by $\big(1 - (k-1)\frac{\sqrt{q}}{q-1}\big)$, we see it is sufficient to show that
\[\big(\ze(k) - \zeta(k+1)\big) - (k-1)\frac{\sqrt{q}}{q-1}\ze(k) - \log \Big( \frac{q}{q-1} \Big) \zeta(k-1)  \ > \  0.\]
From a special case of the principal result of \cite{MR3768303}, we know that 
\[\zeta(k) - \zeta(k+1)  > \  \frac{1}{2^{k+1}}.\]
Using this and the bound $\log \big( \frac{q}{q-1} \big) \leq \frac{1}{q-1}$, this reduces to showing that
\[\frac{1}{2^{k+1}} - (k-1)\frac{\sqrt{q}}{q-1}\ze(k) - \frac{1}{q-1} \zeta(k-1)  \ > \  0.\]
Clearing out the denominators, we find the equivalent statement 
\[(q-1) - 2^{k+1}(k-1)\ze(k)\sqrt{q} - 2^{k+1}\zeta(k-1)  \ > \  0.\]
Letting $a = 2^{k+1}\zeta(k-1)+1$ and $b = (k-1)2^{k+1}\ \!\zeta(k)$, we need that $q - b\sqrt{q} - a > 0$. By the quadratic formula, this inequality will be true for any
\[q  \ \geq \  \frac{\sqrt{4ab^2 + b^4} + 2a + b^2}{2}.\]
Because we seek an upper bound for $q_k$, we can define the constant $\eta$ to equal $a/b^2$ for $k=4$. Because $a/b^2$ decreases exponentially in $k$, this means $4ab^2 \leq 4\eta b^4$ and $2a \leq 2\e b^2$ for all $k \geq 4$. Explicitly, $\eta = \frac{2^5 \ze(3) + 1}{9 \cdot 2^{10} \ze(4)^2} = 0.00365$.... Then
\begin{align*}
    q_k \ &\leq \  \frac{\sqrt{4\eta b^4 + b^4} + 2\eta b^2 + b^2}{2} \ = \  \frac{b^2}{2} \Big( 1 + 2\eta + \sqrt{1 + 4\eta} \Big) \\
& = \ (k-1)^2 4^k \zeta(k)^2 \cdot 2 \Big( 1 + 2\eta + \sqrt{1 + 4\eta} \Big)
\end{align*}
will also be sufficient.  The fact that $\zeta(k)^2 = O(1)$ gives us that $q_k = O(k^2\ \!4^k)$. The constant evaluates to $4.02919... < 4.03$, which gives us the rest of the theorem.
\end{proof}

\section{Computation of $\f(\I_{k,q})$}

As over the integers, the partial sums of $\f(\I_{k,q})$ converge very slowly once $k \geq 2$.  While it is possible to compute these sums using the technique developed in \cite{lichtman} for the sums over the integers with $k$ prime factors, we develop a new method for estimating the size of the tails of these sums after precomputing the counts of polynomials having at most $k$ factors, all less than some degree $N$.  Experimentally, this method is able to compute the values of these sums much faster and with greater precision.  The key idea will be to use the more general form of Theorem \ref{thm:mordell}.  We start with a formula for computing the number of ``smooth'' polynomials (smooth meaning that all of the divisors have degree smaller than some fixed bound) with a fixed number of divisors.

\subsection{Smooth polynomials with $k$ irreducible factors}
Let $\Psi'_{k,q}(n,m)$ denote the count of monic polynomials of degree $n$ with exactly $k$ irreducible factors all of degree at most $m$.  
\begin{theorem} \label{thm:smoothpolys}
We can compute $\Psi'_k(n,N)$ from the values of $\pi'(i)$ by the formula  
\[\Psi'_{k,q}(n,m) \ = \  \sum_{\substack{\ell_1 + 2\ell_2 + \ldots + m\ell_m = n \\ \ell_1 + \ell_2 + \ldots + \ell_m = k}} \prod_{j=1}^m \binom{\ell_j + \pi'_q(j) - 1}{\ell_j},\]
where each $\ell_i$ is a nonnegative integer.
\end{theorem}

\begin{proof}
Fix $n$, $k$ and $m$, and group the polynomials of degree $n$ with $k$ irreducible factors of degree at most $m$ according to the multiplicities of the degrees of their factors. That is, for any polynomial $f$ we define its class by the sequence $\{\ell_j\}$, where $\ell_j := \#\{p \in \I : p|f, \deg p = j\}$. Conversely, each sequence of $\{\ell_j\}$ defines a class of polynomials, which is included in our count if and only if $\{\ell_j\}$ satisfy $\sum_{j=1}^m \ell_j = k$ and $\sum_{j=1}^m j\ell_j = n$.

Now it remains to count the number of polynomials in each such class.  There are $\pi'_q(j)$ monic irreducible polynomials of degree $j$, and we need to choose $\ell_j$ of them with repetition, where order does not matter. The number of ways to select $\ell_j$ such irreducible polynomials of degree $j$ with the potential for repetition is $\binom{\ell_j + \pi'_q(j) - 1}{\ell_j}$.  Multiplying these terms then gives the number of polynomials contained in each class. 
\end{proof}

\subsection{Effective computation of $\mathcal{F}(\mathcal{I}_{k,q})$}  \label{sec:computation}

We can expand on the idea of Section 2.1 to obtain an algorithm which can rapidly compute the value of $\f(\mathcal{I}_{k,q})$ for any $q$ and $k$.  The key idea is to use the full generality of Theorem \ref{thm:mordell} to estimate the size of the tail after approximating with a partial sum.  

In this case, for a fixed $N$ the values of $\pi'_q(n)$ are computed for all $n\leq N$. Let $d(f)$, $D(f)$ and $\Omega(f)$ denote, as before, the degrees of the smallest irreducible factor of $f$, the degree of the largest irreducible factor of $f$ and the total number of irreducible factors of $f$ respectively. We then write 
\begin{align}
    \f(\I_{k,q}) \ = \  \sum_{\substack{f \in \mathcal{I}_{k,q}}} \frac{1}{||f||\deg f} \ &= \ \sum_{\substack{f \in \mathcal{I}_{k,q} \\D(f)\leq N}} \frac{1}{||f||\deg f} + \sum_{\substack{f \in \mathcal{I}_{k,q} \\D(f) > N}} \frac{1}{||f||\deg f}  \label{eq:smoothdecomp}
\end{align} 

The first sum above is computed exactly, using the precomputed values of $\pi'_q(n)$,  as \[S_{k,N,q} \ = \  \sum_{\substack{f \in \mathcal{I}_{k,q} \\D(f)\leq N }} \frac{1}{||f||\deg f} \ = \  \sum_{k \leq n \leq N} \frac{\Psi_{k,q}'(n,N)}{nq^n},\]
where the value of $\Psi_{k,q}'(n,N)$ is computed using Theorem \ref{thm:smoothpolys}.

The second sum will be estimated using a combination of the precomputed values and estimates for the tail using the ``Mordell Sum''
\begin{equation}
    M(k,N,a) \ = \  \sum_{N\leq n_1,n_2,\ldots n_k} \frac{1}{n_1n_2\cdots n_k(n_1+n_2+ \cdots + n_k +a)}.
\end{equation}
Note that when $N=1$, the value of $M(k,1,a)$ is given by Theorem \ref{thm:mordell} as \[M(k,1,a) \ = \  \ k! \sum_{i=0}^\infty \frac{\left(-1\right)^i}{(i+1)^{k+1}} \binom{a-1}{i}, \] which for all $a\geq 0$ will either be a rational number or a multiple of $\zeta(k+1)$.  We can then recursively compute values of this sum for larger values of $N$ using the recurrence 
\begin{align*}
    M(k,N,a) \ &= \ M(k,N-1,a)-\frac{M(k-1,N-1,a+(N-1))}{N-1}+ \cdots \\
    &= \ \sum_{i=0}^k \frac{(-1)^iM(k-i,N-1,a+i(N-1))}{(N-1)^i},
\end{align*} 
which is obtained using an inclusion-exclusion argument over sums where the least allowed term is $(N-1)$.

We start by obtaining an upper bound for the rightmost sum in \eqref{eq:smoothdecomp}. We rewrite this sum as follows,
where $\sum'$ is used to denote that the innermost sum is evaluated over squarefree polynomials.

\begin{align}
  \sum_{\substack{f \in \mathcal{I}_{k,q} \\D(f) > N }} \frac{1}{||f||\deg f} \ &= \ \sum_{i=1}^{k} \sum_{\substack{f \in \mathcal{I}_{k-i,q} \\D(f)\leq N }} \sideset{}{'}\sum_{\substack{g \in \mathcal{I}_{i,q} \\d(g) > N }} \frac{1}{||fg||\deg fg} + \sum_{\substack{f \in \mathcal{I}_{k,q} \\ p^2|f, \ p \in \mathcal{I}_q \\ \deg p >N}}\frac{1}{||f||\deg f} \nonumber \\
  &= \ \sum_{i=1}^{k} \sum_{n\leq(k-i)N} \frac{\Psi'_{k-i,q}(n,N)}{q^n} \sideset{}{'}\sum_{\substack{f \in \mathcal{I}_{i,q} \\d(f) > N }} \frac{1}{||f||(\deg f+n)} + \sum_{\substack{f \in \mathcal{I}_{k,q}\\ p^2|f, \ p \in \mathcal{I}_q \\ \deg p >N}}\frac{1}{||f||\deg f}.\label{eq:sqarefuldecomp}
\end{align}
Consider the left hand sum above.  By the same argument as in the proof of Proposition \ref{prop.piasympgeneral}, we can bound the innermost sum over squarefree polynomials by 
\begin{align}
    \sideset{}{'}\sum_{\substack{f \in \mathcal{I}_{i,q} \\d(f) > N }} \frac{1}{||f||(\deg (f)+n)} \ &\leq \ \frac{1}{i!} \sum_{\substack{N<j_1, \ldots, j_{i}}} \frac{\pi'_q(j_1)\pi'_q(j_2)\ldots\pi'_q(j_{i})}{q^{j_1+j_2+\cdots+j_{i}}(j_1+j_2+\cdots+j_{i}+n)}\nonumber \\
    &\leq \ \frac{1}{i!} \sum_{\substack{N<j_1, \ldots, j_{i}}} \frac{q^{j_1}q^{j_2}\cdots q^{j_{i}}}{q^{j_1+j_2+\cdots+j_{i}}(j_1j_2\cdots j_i)(j_1+j_2+\cdots+j_{i}+n)} \nonumber\\
    &\leq \ \frac{1}{i!} \sum_{\substack{N<j_1, \ldots, j_{i}}} \frac{1}{(j_1j_2\cdots j_i)(j_1+j_2+\cdots+j_{i}+n)}\ = \ \frac{1}{i!} M(i,N+1,n).\label{eq:squarefreebd} 
\end{align}
Here we have used again Proposition \ref{lower bound of pi'} to obtain an upper bound for $\pi'(j)$.  We then bound the rightmost sum of \eqref{eq:sqarefuldecomp} over ``squarefull'' polynomials by
\begin{align*}
    \sum_{\substack{f \in \mathcal{I}_{k,q} \\ p^2|f, \ p \in \mathcal{I}_q \\ \deg p>N}}\frac{1}{||f||\deg f} \ &\leq \sum_{\substack{p \in \mathcal{I}_q \\ \deg p > N}} \frac{1}{q^{2\deg p}} \sum_{\substack{f \in \mathcal{I}_{k-2,q}}} \frac{1}{||f||\deg f}\\
    & = \ \mathcal{F}(\mathcal{I}_{k-2,q})\sum_{N<n} \frac{\pi'_q(n)}{q^{2n}}  \ < \  2\sum_{N<n} \frac{1}{nq^{n}} \ < \ \frac{2}{Nq^{N+1}(1-1/q)} \ < \ \frac{2}{Nq^N}.\\
\end{align*}
Using this and \eqref{eq:squarefreebd} in \eqref{eq:sqarefuldecomp}, we get the bound 
\begin{align}
    \sum_{\substack{f \in \mathcal{I}_{k,q} \\D(f) > N }} \frac{1}{||f||\deg f}  \ &\leq \ \sum_{i=1}^{k} \sum_{n\leq(i-1)N} \frac{1}{i!q^n}\Psi'_{k-i,q}(n,N) M(i,N+1,n) + \frac{2}{Nq^N} \nonumber \\
    &= \ R_{k,N,q} + \frac{2}{Nq^N}.
\end{align}
We can similarly get a lower bound.  We start by writing
\begin{align}
    \sum_{\substack{f \in \mathcal{I}_{k,q} \\D(f) > N }} \frac{1}{||f||\deg f} \ &= \ \sum_{i=1}^{k} \sum_{\substack{f \in \mathcal{I}_{k-i,q} \\D(f)\leq N }}\sum_{\substack{g \in \mathcal{I}_{i,q} \\d(g) > N }} \frac{1}{||fg||\deg fg}\nonumber \\
    &= \ \sum_{i=1}^{k} \sum_{n\leq(k-i)N} \frac{\Psi'_{k-i,q}(n,N)}{q^n} \sum_{\substack{g \in \mathcal{I}_{k,q} \\d(g) > N }} \frac{1}{||g||(\deg g +n)} \nonumber \\
    &\geq \ \sum_{i=1}^{k} \sum_{n\leq(k-i)N} \frac{\Psi'_{k-i,q}(n,N)}{q^n} \frac{1}{i!} \sum_{\substack{N<j_1, \ldots, j_{i}}} \frac{\pi'_q(j_1)\pi'_q(j_2)\ldots\pi'_q(j_{i})}{q^{j_1+j_2+\cdots+j_{i}}(j_1+j_2+\cdots+j_{i}+n)}. \label{eq:roughlowerbd}
\end{align}
This time we don't restrict to squarefree polynomials, and the argument for the lower bound is the same as that of Proposition \ref{prop.pi_k bound}. We use Proposition \ref{lower bound of pi'} to bound this innermost sum from below as
\begin{align}
    \sum_{\substack{N<j_1, \ldots, j_{i}}} \frac{\pi'_q(j_1)\pi'_q(j_2)\ldots\pi'_q(j_{i})}{q^{j_1+j_2+\cdots+j_{i}}(j_1+j_2+\cdots+j_{i}+n)} \ &\geq \ \sum_{\substack{N<j_1, \ldots, j_{i}}} \frac{(q^{j_1} - \frac{q}{q-1}\cdot q^{j_1/2}) ...(q^{j_i} - \frac{q}{q-1} \cdot q^{j_i/2} )}{q^{j_1+j_2+\cdots+j_{i}}j_1 ... j_k(j_1+j_2+\cdots+j_{i}+n) } \nonumber \\
    & = \
     \sum_{\substack{N<j_1, \ldots, j_{i}}} \frac{(1 - \frac{q}{q-1}\cdot q^{-j_1/2}) ...(1 - \frac{q}{q-1} \cdot q^{-j_i/2} )}{j_1 ... j_i(j_1+j_2+\cdots+j_{i}+n) } \nonumber \\
    &> \ \sum_{\substack{N<j_1, \ldots, j_{i}}} \frac{1 - i\frac{q}{q-1}\cdot q^{-N/2}}{j_1 ... j_i(j_1+j_2+\cdots+j_{i}+n) } \nonumber \\
    &= \ \left(1 - i\frac{q}{q-1}\cdot q^{-N/2}\right)M(i,N+1,n).
\end{align}
Inserting this in \eqref{eq:roughlowerbd} gives
\begin{align}
    \sum_{\substack{f \in \mathcal{I}_{k,q} \\D(f) > N}} \frac{1}{||f||\deg f}  \ &\geq \ \sum_{i=1}^{k}\left(1 - \frac{iq^{1-N/2}}{q-1}\right) \sum_{n\leq(i-1)N} \frac{1}{i!q^n}\Psi'_{k-i,q}(n,N) M(i,N+1,n) \nonumber \\
    &= \ R_{k,N,q} - \ \frac{q^{1-N/2}}{q-1}\sum_{i=1}^{k} \sum_{n\leq(i-1)N} \frac{\Psi'_{k-i,q}(n,N) M(i,N+1,n)}{(i-1)!q^n}. \nonumber
\end{align}
Thus we can approximate $\mathcal{F}(\mathcal{I}_{k,q}) \approx S_{k,N,q}+R_{k,N,q}$, with the effective bounds
\begin{equation}
    \frac{q^{1-N/2}}{q-1}\sum_{i=1}^{k} \sum_{n\leq(i-1)N} \frac{\Psi'_{k-i,q}(n,N) M(i,N+1,n)}{(i-1)!q^n}  \ \leq \  \mathcal{F}(\mathcal{I}_{k,q}) - S_{k,N,q} - R_{k,N,q}  \ < \   \frac{2}{Nq^N}. \label{eq:finalnumbounds}
\end{equation} 

\subsection{Numerical Computations} \label{sec:numresults}
Code was written in C++ to compute the bounds in \eqref{eq:finalnumbounds} using MPFR for multiple-precision floating-point computations with correct rounding.  Using $N=80$, we have computed the values of $\mathcal{F}(\mathcal{I}_{k,q})$  for all $k\leq 64$ and $2\leq q<64$ to at least ten decimal places, and we have computed various special cases with more accuracy to higher values of $k$. The results for $q<7$ are documented in the tables of Appendix \ref{app:data}.

From these computations we find that the analogue of the Banks-Martin conjecture is false for $q=2$, $3$, and $4$.  In particular, we see that the sequence $\mathcal{F}(\I_{k,2})$ has a local minimum at $k=4$ of $\mathcal{F}(\I_{4,2})=0.956237\ldots$. Similarly, $\mathcal{F}(\I_{k,3})$ has a local minimum at $k = 6$ of $\mathcal{F}(\I_{6,3})=0.994968\ldots$ and $\mathcal{F}(\I_{k,4})$ has a local minimum at $k=9$ of $\mathcal{F}(\I_{9,4})= 0.999781\ldots$.

It seems highly likely that each of these values is in fact a global minimum, since the values appear to increase monotonically to 1 as $k \to \infty$. In fact, it appears that $1-\mathcal{F}(\I_{k,q}) = O(2^{-k})$ in each case.  This agrees with the observations of Lichtman in the integers, where a global minimum in the sum $\mathcal{F}(\mathcal{P}_k)$ was observed at $k=6$, with similar convergence to 1 as $k \to \infty$ \cite{lichtman}.

Surprisingly for $q\geq 5$, the behaviour appears to be quite different, with the numerical evidence suggesting that the values of $\mathcal{F}(\I_{k,q})$ are monotonically decreasing to 1 as $k \to \infty$.  Again we see the the same $O(2^{-k})$ convergence, just from above.  For $q=5$ we have verified that \[\mathcal{F}(\I_{k,5}) \ > \ \mathcal{F}(\I_{k+1,5}) \ > \ 1\] for all $k<100$.   Based on this evidence, we conjecture that the analogue of the Banks-Martin conjecture over $\F_q[x]$ is still true in these cases.

\begin{conjecture}\label{fut.conj}
For each $q\geq 5$ the inequalities
$$\f(\mathcal{I}_{1,q})  \ > \  \f(\mathcal{I}_{2,q})  \ > \  \ldots  \ > \  \f(\mathcal{I}_{k,q})  \ > \  \f(\mathcal{I}_{k+1,q})\ldots$$
hold for all positive integers $k$.
\end{conjecture}

\renewcommand{\biblistfont}{\normalfont\normalsize}
\bibliographystyle{amsplain}
\bibliography{sources}

\begin{bibdiv}
\begin{biblist}

\bib{MR3768303}{article}{
      author={Ballantine, Cristina},
      author={Merca, Mircea},
       title={Finite differences of {E}uler's zeta function},
        date={2017},
        ISSN={1787-2405},
     journal={Miskolc Math. Notes},
      volume={18},
      number={2},
       pages={639\ndash 642},
         url={https://doi.org/10.18514/mmn.2017.2256},
      review={\MR{3768303}},
}

\bib{banksMartin13}{article}{
      author={Banks, William~D.},
      author={Martin, Greg},
       title={Optimal primitive sets with restricted primes},
        date={2013},
        ISSN={1553-1732},
     journal={Integers},
      volume={13},
       pages={Paper No. A69, 10},
      review={\MR{3118387}},
}

\bib{bkk19}{article}{
      author={Bayless, Jonathan},
      author={Kinlaw, Paul},
      author={Klyve, Dominic},
       title={Sums over primitive sets with a fixed number of prime factors},
        date={2019},
        ISSN={0025-5718},
     journal={Math. Comp.},
      volume={88},
      number={320},
       pages={3063\ndash 3077},
         url={https://doi.org/10.1090/mcom/3416},
      review={\MR{3985487}},
}

\bib{cohen}{article}{
      author={Cohen, Henri},
       title={High precision computation of {H}ardy-{L}ittlewood constants},
     journal={preprint},
         url={https://www.math.u-bordeaux.fr/~hecohen/},
}

\bib{erdos35}{article}{
      author={Erd\H{o}s, Paul},
       title={Note on sequences of integers no one of which is divisible by any
  other},
        date={1935},
     journal={J. London Math. Soc.},
      volume={10},
      number={2},
       pages={126\ndash 128},
         url={https://doi.org/10.1112/jlms/s1-10.1.126},
      review={\MR{1574239}},
}

\bib{erdzhang94}{article}{
      author={Erd\H{o}s, Paul},
      author={Zhang, Zhen~Xiang},
       title={Upper bound of {$\sum 1/(a_i\log a_i)$} for primitive sequences},
        date={1993},
        ISSN={0002-9939},
     journal={Proc. Amer. Math. Soc.},
      volume={117},
      number={4},
       pages={891\ndash 895},
         url={https://doi.org/10.2307/2159512},
      review={\MR{1116257}},
}

\bib{us19}{article}{
      author={G\'{o}mez-Colunga, Andr\'{e}s},
      author={Kavaler, Charlotte},
      author={McNew, Nathan},
      author={Zhu, Mirilla},
       title={On the size of primitive sets in function fields},
        date={2020},
        ISSN={1071-5797},
     journal={Finite Fields Appl.},
      volume={64},
       pages={101658, 23},
         url={https://doi.org/10.1016/j.ffa.2020.101658},
      review={\MR{4078938}},
}

\bib{lichtman}{article}{
      author={Lichtman, Jared~Duker},
       title={Almost primes and the {B}anks-{M}artin conjecture},
        date={2020},
        ISSN={0022-314X},
     journal={J. Number Theory},
      volume={211},
       pages={513\ndash 529},
         url={https://doi.org/10.1016/j.jnt.2019.11.006},
      review={\MR{4074566}},
}

\bib{lichtmanPom19}{article}{
      author={Lichtman, Jared~Duker},
      author={Pomerance, Carl},
       title={The {E}rd\h{o}s conjecture for primitive sets},
        date={2019},
        ISSN={2330-1511},
     journal={Proc. Amer. Math. Soc. Ser. B},
      volume={6},
       pages={1\ndash 14},
         url={https://doi.org/10.1090/bproc/40},
      review={\MR{3937344}},
}

\bib{MR0100181}{article}{
      author={Mordell, L.~J.},
       title={On the evaluation of some multiple series},
        date={1958},
        ISSN={0024-6107},
     journal={J. London Math. Soc.},
      volume={33},
       pages={368\ndash 371},
         url={https://doi.org/10.1112/jlms/s1-33.3.368},
      review={\MR{0100181}},
}

\bib{MR2712407}{book}{
      author={Pollack, Paul},
       title={Prime polynomials over finite fields},
   publisher={ProQuest LLC, Ann Arbor, MI},
        date={2008},
        ISBN={978-0549-85231-5},
        note={Thesis (Ph.D.)--Dartmouth College},
      review={\MR{2712407}},
}

\bib{harmonic}{book}{
      author={P\'{o}lya, George},
      author={Szeg\H{o}, Gabor},
       title={Problems and theorems in analysis. {I}},
      series={Classics in Mathematics},
   publisher={Springer-Verlag, Berlin},
        date={1998},
        ISBN={3-540-63640-4},
         url={https://doi.org/10.1007/978-3-642-61905-2},
      review={\MR{1492447}},
}

\bib{MR1700882}{article}{
      author={Rosen, Michael},
       title={A generalization of {M}ertens' theorem},
        date={1999},
        ISSN={0970-1249},
     journal={J. Ramanujan Math. Soc.},
      volume={14},
      number={1},
       pages={1\ndash 19},
      review={\MR{1700882}},
}

\bib{rosen_2011}{book}{
      author={Rosen, Michael},
       title={Number theory in function fields},
      series={Graduate Texts in Mathematics},
   publisher={Springer-Verlag, New York},
        date={2002},
      volume={210},
        ISBN={0-387-95335-3},
         url={https://doi.org/10.1007/978-1-4757-6046-0},
      review={\MR{1876657}},
}

\bib{harmonic2}{article}{
      author={Young, Robert~M.},
       title={Euler's constant},
        date={1991},
        ISSN={00255572},
     journal={The Mathematical Gazette},
      volume={75},
      number={472},
       pages={187\ndash 190},
}

\bib{zhang91}{article}{
      author={Zhang, Zhen~Xiang},
       title={On a conjecture of {E}rd\h{o}s on the sum {$\sum_{p\leq
  n}1/(p\log p)$}},
        date={1991},
        ISSN={0022-314X},
     journal={J. Number Theory},
      volume={39},
      number={1},
       pages={14\ndash 17},
         url={https://doi.org/10.1016/0022-314X(91)90030-F},
      review={\MR{1123165}},
}

\end{biblist}
\end{bibdiv}
\nocite{*}

\pagebreak

\appendix
\label{app:data}
\section{Numerical Data on $\f(\I_{k,q})$}

\scriptsize

\input apnum
\apFRAC=50
\def\outputdigs {8}

\def\ep#1{$
\def\origval {#1}
\apDIG\origval\apnumA
\def\mynumD{\apnumA}
\evaldef\newval{#1+0}
\apDIG\newval\apnumB
\def\mynewnumD{\apnumB}
\newcount\myoutputdigs
\myoutputdigs=\apnumA
\evaldef\foo{#1-1}
\ifnum\apSIGN>0 
  \def\mysign {+} 
  \evaldef\nfo{\iFLOOR{\LN{#1-1}/\LN{10}}}
\else 
  \def\mysign {}
  \evaldef\nfo{\iFLOOR{\LN{1-#1}/\LN{10}}}
\fi
\advance\myoutputdigs\nfo
\evaldef\nfoo{\foo/10^{\nfo}}
\ifnum\myoutputdigs>\outputdigs
   \apROUND\nfoo{\outputdigs}
   \ifnum\mynumD>\mynewnumD
      \apnumZ=\outputdigs
      \advance\apnumZ-\mynewnumD
      \advance\apnumZ-\nfo
      \ifnum\apnumZ>0
         \apADDzeros\nfoo
      \fi
   \fi
\else
   \apnumZ=\mynumD
   \advance\apnumZ-\mynewnumD
   \ifnum\apnumZ>0
      \apADDzeros\nfoo
   \fi
\fi
\hspace{0mm}
1 \mysign \nfoo \times 10^{\nfo}$}

\begin{table}[h!]
    \centering
    \hspace{-5mm}\begin{tabular}{c|l|l|l|l|l}
$k$ &$\f(\I_{k,2})$&$\f(\I_{k,3})$& $\f(\I_{k,4})$& $\f(\I_{k,5})$&$\f(\I_{k,7})$ \\
\hline
1 & 1.4676602238442289268 & 1.5402654962770992783& 1.5708306089585806605& 1.5876369878229405564 & 1.6055616864329830894 \\
2 & 1.0644425954143168595& 1.1301714500071343633& 1.1544864845853626474&  1.1668343411440889017 & 1.1790969073890668757 \\
3 & 0.9755638525263773555& 1.0329809138654179703& 1.0517959064091933064&  1.0606722482320695710 & 1.0689297642298799167 \\
4 & 0.9562373433151932108& 1.0039698809027713378& 1.0178327413536777409&  1.0239276909306761841 & 1.0292662613922721641 \\
5 & 0.9581408226316153830& 0.9960179423616558785& 1.0057528618201179388&  1.0097501408648004439 & 1.0130607223966467259 \\
6 & 0.9661285846774159333& 0.9949687972770260308& 1.0015148661835156763&  1.0040299319147160468 & 1.0060072704223504918 \\
7 & 0.9747368549520022143& 0.9959150552841082468& 1.0001513629475453519&  1.0016773165460739756 & 1.0028205606817957574 \\
8 & 0.9820875563671306239& 0.9971537408436136635& 0.9998044985849281472&  $1.0007015961030813973$ & 1.0013445588428900262 \\
9 & 0.9877477647269600411& 0.9981715655684219998& 0.9997818901532824166&  $1.0002951481314120617$ & 1.0006484376681192577 \\
10& 0.9918478580517178761& 0.9988850772260466434& 0.9998382721719850807&  $1.0001251569695533427$ & 1.0003155548064037100 \\
11& 0.9946958995719092591& 0.9993449618001374514& 0.9998964608075082070&  $1.0000536550878474023$ & 1.0001546392373837702 \\
12& 0.9966129963802004602& 0.9996258781376391880& 0.9999386341284465735&  $1.0000233235195754782$ & 1.0000761897594317088 \\
13& 0.9978716414731367847& 0.9997910525523849739& 0.9999653236912439819&  $1.0000103050740808725$ & 1.0000376910537044268 \\
14& 0.9986811586590295149& 0.9998854079068883995& 0.9999810267336057817&  $1.0000046345106376206$ & 1.0000187023114699499 \\
15& 0.9991928200040364980& 0.9999380941647140707& 0.9999898540553063995&  $1.0000021220772785316$ & 1.0000093007543809601 \\
16& 0.9995113928362548792& 0.9999669753501071136& 0.9999946649895369026&  $1.0000009884203117771$ & 1.0000046327688406045 \\
17& 0.9997071495459197712& 0.9999825683197510197& 0.9999972297481776256&  $1.0000004675067009067$ & 1.0000023102741326004 \\
18& 0.9998260439229729488& 0.9999908809671412008& 0.9999985750920124071&  $1.0000002240397871479$ & 1.0000011530296441410 \\
19& 0.9998975071335135411& 0.9999952655907561082& 0.9999992723231812274&  $1.0000001085215464210$ & 1.0000005757919164950 \\
20& 0.9999400604052216966& 0.9999975577054795509& 0.9999996304000129789&  $1.0000000530122567057$ & 1.0000002876491601127 \\
21& 0.9999651849121454024& 0.9999987469201610398& 0.9999998130414889626&  $1.0000000260642202254$ & $1.0000001437406709444$ \\
22& 0.9999799048973557495& 0.9999993600015746679& 0.9999999057197182304&  $1.0000000128768650727$ & $1.0000000718419078680$ \\
23& 0.9999884683698132380& 0.9999996743758878555& 0.9999999525652106483&  $1+6.38426063\times 10^{-9}$  & $1.0000000359113517749$ \\
24& 0.9999934180456298710& 0.9999998348554963871& 0.9999999761751077814&  $1+3.17334132\times 10^{-9}$  & $1.0000000179524406680$ \\
25& 0.9999962619151749007& 0.9999999164676982231& 0.9999999880486329738&  $1+1.58018897\times 10^{-9}$  & $1+8.97513258\times 10^{-9}$ \\
26& 0.9999978868847903109& 0.9999999578414305089& \ep{0.99999999401032400167805170343356478}&  $1+7.87869180\times 10^{-10}$ & $1+4.48720120\times 10^{-9}$ \\
27& 0.9999988106438072540& 0.9999999787613648921& \ep{0.99999999700016446390809329788146931}&  $1+3.93173826\times 10^{-10}$ & $1+2.24347823\times 10^{-9}$ \\
28& 0.9999993332889212801& 0.9999999893163745066& \ep{0.99999999849830814589668351627982141}&  $1+1.96327799\times 10^{-10}$ & $1+1.12169815\times 10^{-9}$ \\
29& 0.9999996276842568036& \ep{0.99999999463238650619207000903880746}& \ep{0.99999999924852782534224093518011180}&  $1+9.80759342\times 10^{-11}$ & $1+5.60835373\times 10^{-10}$ \\
30& 0.9999997928264723852& \ep{0.99999999730589904168365258469840121}& \ep{0.99999999962404399843442291708942366}&  $1+4.90081887\times 10^{-11}$ & $1+2.80413107\times 10^{-10}$ \\
31& 0.9999998851056604170& \ep{0.99999999864886375119672981046174231}& \ep{0.99999999981194514488401094265328192}&  $1+2.44940385\times 10^{-11}$ & $1+1.40205024\times 10^{-10}$ \\
32& 0.9999999364831266058& \ep{0.99999999932281828860672310974899083}& \ep{0.99999999990594583299411194145404890}&  $1+1.22436306\times 10^{-11}$ & $1+7.01020012\times 10^{-11}$ \\
33& 0.9999999649907064879& \ep{0.99999999966077489794753525400280151}& \ep{0.99999999995296365125080805104867327}&  $1+6.12067553\times 10^{-12}$ & $1+3.50508301\times 10^{-11}$ \\
34& 0.9999999807578865298& \ep{0.99999999983013926573944592127047924}& \ep{0.99999999997647862729019710253853186}&  $1+3.05995497\times 10^{-12}$ & $1+1.75253582\times 10^{-11}$ \\
35& 0.9999999894521671497& \ep{0.99999999991497293195355973684855028}& \ep{0.99999999998823821340708499185386399}&  $1+1.52984909\times 10^{-12}$ & $1+8.76266014\times 10^{-12}$ \\
36& \ep{0.9999999942326596671}& \ep{0.99999999995744894095973019799852647}& \ep{0.99999999999411872941836848495782572}&  $1+7.64881531\times 10^{-13}$ & $1+4.38132374\times 10^{-12}$ \\
37& \ep{0.9999999968540914682}& \ep{0.99999999997870996520318808855654319}& \ep{0.99999999999705923570908588386669822}&  $1+3.82426367\times 10^{-13}$ & $1+2.19065976\times 10^{-12}$ \\
38& \ep{0.9999999982879085152}& \ep{0.99999999998934939631452506999770194}& \ep{0.99999999999852957386426337321610345}&  $1+1.91208367\times 10^{-13}$ & $1+1.09532917\times 10^{-12}$ \\
39& \ep{0.9999999990702489384}& \ep{0.99999999999467255435369961772927814}& \ep{0.99999999999926477196730520477595758}&  $1+9.56025742\times 10^{-14}$ & $1+5.47664354\times 10^{-13}$ \\
40& \ep{0.99999999949613947874}& \ep{0.99999999999733545724245376337332318}& \ep{0.99999999999963238090404814081501463}&  $1+4.78007493\times 10^{-14}$ & $1+2.73832098\times 10^{-13}$ \\
41& \ep{0.9999999997274790292}& \ep{0.99999999999866741603538373041696105}& \ep{0.99999999999981618873126651831425969}&  $1+2.39001951\times 10^{-14}$ & $1+1.36916023\times 10^{-13}$ \\
42& \ep{0.99999999985287899831}& \ep{0.99999999999933358921796384316241914}& \ep{0.99999999999990809378376717665964599}&  $1+1.19500376\times 10^{-14}$ & $1+6.84580029\times 10^{-14}$ \\
43& \ep{0.9999999999207186197}& \ep{0.99999999999966674959226351870846820}& \ep{0.99999999999995404669544992810778873}&  $1+5.97499880\times 10^{-15}$ & $1+3.42289985\times 10^{-14}$ \\
44& \ep{0.99999999995734953268}& \ep{0.99999999999983335778639943694981556}& \ep{0.99999999999997702328150812465251951}&  $1+2.98749272\times 10^{-15}$ & $1+1.71144983\times 10^{-14}$ \\
45& \ep{0.99999999997709319249}& \ep{0.99999999999991667248351965828718754}& \ep{0.99999999999998851161846225289670007}&  $1+1.49374413\times 10^{-15}$ & $1+8.55724884\times 10^{-15}$ \\
46& \ep{0.99999999998771645696}& \ep{0.99999999999995833383275544077545592}& \ep{0.99999999999999425580173553395343683}&  $1+7.46871326\times 10^{-16}$ & $1+4.27862431\times 10^{-15}$ \\
47& \ep{0.9999999999934229771010}& \ep{0.9999999999999791660132572011333951}& \ep{0.99999999999999712789835006184055815}&  $1+3.73435415\times 10^{-16}$ & $1+2.13931212\times 10^{-15}$ \\
48& \ep{0.99999999999648351887168}& \ep{0.9999999999999895826688728280150586}& \ep{0.99999999999999856394833015718290294}&  $1+1.86717625\times 10^{-16}$ & $1+1.06965604\times 10^{-15}$ \\
49& \ep{0.9999999999981224758858}& \ep{0.9999999999999947912084137655495553}& \ep{0.99999999999999928197388180102780763}&  $1+9.33587851\times 10^{-17}$ & $1+5.34828020\times 10^{-16}$ \\
50& \ep{0.9999999999989988833782}& \ep{0.9999999999999973955572908956584618}& \ep{0.99999999999999964098684599155750495}&  $1+4.66793833\times 10^{-17}$ & $1+2.67414008\times 10^{-16}$ \\
51& \ep{0.999999999999466876258444}& \ep{0.9999999999999986977612171386759606}& \ep{0.99999999999999982049339121868233063}&  $1+2.33396886\times 10^{-17}$ & $1+1.33707004\times 10^{-16}$ \\
52& \ep{0.999999999999716445445625}& \ep{0.9999999999999993488741477788880708}& \ep{0.99999999999999991024668497606876226}&  $1+1.16698432\times 10^{-17}$ & $1+6.68535018\times 10^{-17}$ \\
53& \ep{0.999999999999849363441523}& \ep{0.9999999999999996744346836346499651}& \ep{0.99999999999999995512333893175611574}&  $1+5.83492130\times 10^{-18}$ & $1+3.34267508\times 10^{-17}$ \\
54& \ep{0.999999999999920066628709}& \ep{0.9999999999999998372164592224338230}& \ep{0.99999999999999997756166827702593676}&  $1+2.91746053\times 10^{-18}$ & $1+1.67133754\times 10^{-17}$ \\  
55& \ep{0.999999999999957631049883}& \ep{0.9999999999999999186079043245177889}& \ep{0.99999999999999998878083374124128564}&  $1+1.45873023\times 10^{-18}$ & $1+8.35668770\times 10^{-18}$ \\
56& \ep{0.999999999999977566036195}& \ep{0.9999999999999999593038324919292042}& \ep{0.99999999999999999439041673791274082}&  $1+7.29365103\times 10^{-19}$ & $1+4.17834385\times 10^{-18}$ \\
57& \ep{0.999999999999988133596385}& \ep{0.9999999999999999796518722970424464}& \ep{0.99999999999999999719520832463892943}&  $1+3.64682547\times 10^{-19}$ & $1+2.08917192\times 10^{-18}$ \\
58& \ep{0.999999999999993729498912}& \ep{0.9999999999999999898259200355157356}& \ep{0.99999999999999999859760414752366627}&  $1+1.82341272\times 10^{-19}$ & $1+1.04458596\times 10^{-18}$ \\
59& \ep{0.999999999999996689680625}& \ep{0.9999999999999999949129541198861500}& \ep{0.99999999999999999929880206882324160}&  $1+9.11706357\times 10^{-20}$ & $1+5.22292981\times 10^{-19}$ \\
60& \ep{0.999999999999998254033763}& \ep{0.9999999999999999974564749045460284}& \ep{0.99999999999999999964940103276352630}&  $1+4.55853177\times 10^{-20}$ & $1+2.61146490\times 10^{-19}$ \\
61& \ep{0.999999999999999079946247}& \ep{0.9999999999999999987282366657812932}& \ep{0.99999999999999999982470051583185883}&  $1+2.27926587\times 10^{-20}$ & $1+1.30573245\times 10^{-19}$ \\
62& \ep{0.999999999999999515588468}& \ep{0.9999999999999999993641180463299146}& \ep{0.99999999999999999991235025773247489}&  $1+1.13963293\times 10^{-20}$ & $1+6.52866226\times 10^{-20}$ \\
63& \ep{0.999999999999999745169403}& \ep{0.9999999999999999996820589189056896}& \ep{0.99999999999999999995617512880504256}&  $1+5.69816468\times 10^{-21}$ & $1+3.26433113\times 10^{-20}$\\
64& \ep{0.999999999999999866052115068}& \ep{0.9999999999999999998410294215730967}& \ep{0.99999999999999999997808756438211072}&  $1+2.84908234\times 10^{-21}$ & $1+1.63216556\times 10^{-20}$\\
65& \ep{0.999999999999999929647717163}& \ep{0.9999999999999999999205146970425831}& \ep{0.99999999999999999998904378218424836}&  $1+1.42454116\times 10^{-21}$ & $1+8.16082782\times 10^{-21}$\\
66& \ep{0.999999999999999963077683695}& \ep{0.9999999999999999999602573435410900}& \ep{0.99999999999999999999452189108985418}&  $1+7.12270584\times 10^{-22}$ & $1+4.08041391\times 10^{-21}$\\
67& \ep{0.999999999999999980636761142}& \ep{0.9999999999999999999801286699682398}& \ep{0.99999999999999999999726094554417013}&  $1+3.56135292\times 10^{-22}$ & $1+2.04020695\times 10^{-21}$\\
68& \ep{0.999999999999999989852597987}& \ep{0.9999999999999999999900643343326832}& \ep{0.99999999999999999999863047277183264}&  $1+1.78067646\times 10^{-22}$ & $1+1.02010347\times 10^{-21}$\\
69& \ep{0.999999999999999994685913715}& \ep{0.9999999999999999999950321669311651} & \ep{0.99999999999999999999931523638583214}&  $1+8.90338230\times 10^{-23}$& $1+5.10051739\times 10^{-22}$\\
70& \ep{0.999999999999999997218955756}& \ep{0.9999999999999999999975160833807801}& \ep{0.99999999999999999999965761819288798}&  $1+4.45169115\times 10^{-23}$& $1+2.55025869\times 10^{-22}$\\
    \end{tabular}
    \caption{\small Values of $\f(\I_{k,q})$ computed using the algorithm in Section \ref{sec:computation}.  To obtain these values we used $N=200$ ($q=2$), $N=150$ ($q=3,4$) and $N=110$ ($q=5,7$), and a precision of 256 bits.  Each number is accurate to as many decimal places as displayed and is truncated (not rounded).\vspace{-5mm}}
    \label{tab:my_label}
\end{table}
\vspace{-15mm}

\normalsize

\end{document}